\documentclass[11pt,reqno]{amsart} 
\usepackage[utf8]{inputenc}
\usepackage{enumitem,amsmath, amsthm, amssymb, amsfonts, mathtools,fullpage} %maths stuff

\usepackage{graphicx, tikz} %graphics stuff
\usepackage{hyperref} %to do hyperlinks
\usepackage{multirow} %this helps us do tables
\usepackage{epigraph}
\usepackage{blindtext}
\usepackage{mathtools}
\usepackage{dirtytalk}
\usepackage{fullpage}
\usepackage{subfig}
\usepackage{color}
\usepackage{tikz}
\usepackage{adjustbox}
\usepackage[normalem]{ulem}
\usepackage{tikz}
\usepackage{graphicx}
\usepackage{enumitem}
\usepackage{float}
\usetikzlibrary{patterns}
\theoremstyle{definition}

\newtheorem{theorem}{Theorem}[section]
\newtheorem{lemma}{Lemma}[section]

\newtheorem{corollary}{Corollary}[section]

\newtheorem{definition}{Definition}[section]

\newtheorem{example}{Example}[section]
\newtheorem{remark}{Remark}[section]
\newtheorem{proposition}{Proposition}[section]

\newcommand{\PF}{\text{PF}}
\newcommand{\MVP}{\text{MVP}}
\newcommand{\Out}{\mathcal{O}}
\newcommand{\M}{\mathcal{M}}

\newcommand{\OPFn}{\mathcal{O}_{\text{PF}_{n}}}

\newcommand{\OMVPn}{\mathcal{O}_{\text{MVP}_{n}}}
\newcommand{\OMVPm}{\mathcal{O}_{\text{MVP}_{m}}}

\newcommand{\Sym}{\mathfrak{S}}

\title{On the outcome map of MVP parking functions: permutations avoiding 321 and 3412, and Motzkin Paths}

\author{Pamela E. Harris}
\thanks{P.~E.~Harris was supported through a Karen Uhlenbeck EDGE Fellowship.}

\author{Brian M. Kamau}
\thanks{B.~M.~Kamau 
was supported through Williams College Science Center.}

\address[P.~E. Harris]{Department of Mathematical Sciences, University of Wisconsin-Milwaukee, Milwaukee, WI 53211}
\email{\textcolor{blue}{\href{mailto:peharris@uwm.edu}{peharris@uwm.edu}}}
\address[B.~M.~Kamau]{Department of Mathematics and Statistics, Williams College, Williamstown, MA 01267}
\email{\textcolor{blue}{\href{mailto:bmk5@williams.edu}{bmk5@williams.edu}}}
\author[Mart\'{i}nez Mori]{J. Carlos Mart\'{i}nez Mori}
\address[J.~C.~Mart\'{i}nez Mori]{Center for Applied Mathematics, Cornell University, Ithaca, NY 14853}
\email{\textcolor{blue}{\href{mailto:jm2638@cornell.edu}{jm2638@cornell.edu}}}

\author[Tian]{Roger Tian}
\address[R.~Tian]{Independent Researcher, Davis, CA 95618}
\email{\textcolor{blue}{\href{mailto:rgtian@ucdavis.edu}{rgtian@ucdavis.edu}}}

\date{\today}

\begin{document}

\begin{abstract}
We introduce a new parking procedure called \emph{MVP parking} in which $n$ cars sequentially enter a one-way street with a preferred parking spot from the $n$ parking spots on the street. If their preferred spot is empty, they park there. Otherwise, they park there and the car parked in that spot is bumped to the next unoccupied spot on the street. If all cars can park under this parking procedure, we say the list of preferences of the $n$ cars is an \emph{MVP parking function} of length $n$. We show that the set of (classical) parking functions is exactly the set of MVP parking functions although the parking outcome (order in which the cars park) is different under each parking process. 
Motivating the question: Given a permutation describing the outcome of the MPV parking process, what is the number of MVP parking functions resulting in that given outcome? Our main result establishes a bound for this count which is tight precisely when the permutation describing the parking outcome avoids the patterns 321 and 3412. 
We then consider special cases of permutations and give closed formulas for the number of MVP parking functions with those outcomes. In particular, we show that the number of MVP parking functions which park in reverse order (that is the permutation describing the outcome is the longest word in $\mathfrak{S}_n$, which does not avoid the pattern 321) is given by the $n$th Motzkin number. 
We also give families of permutations describing the parking outcome for which the cardinality of the set of cars parking in that order is exponential and others in which it is linear.
\end{abstract}

\maketitle


%include 3-5 keywords and include 2020 Mathematics Subject Classification

\section{Introduction}

Throughout we let $\mathbb{N}=\{1,2,3,\ldots\}$ and, for $n\in\mathbb{N}$, we denote $[n]\coloneqq\{1,2,\ldots,n\}$. 
Consider a one-way street consisting of $n$ parking spots enumerated from 1 to $n$. 
There are $n$ cars lined up to enter the street in order, each of which has a preferred parking spot.
Let $\alpha = (a_1,\ldots,a_n) \in [n]^n$ encode the parking spot preferences, where $a_i$ is the preferred spot of car $i$. 
The parking process begins when car 1 enters the street and parks in its preferred spot $a_1$\textemdash as it is the first car to enter the street, it finds spot $a_1$ unoccupied. 
Next, car 2 enters the street and attempts to park in its preferred spot $a_2$. 
If spot $a_2$ is unoccupied, it parks there.
Otherwise, it continues driving down the one-way street and parks in the first unoccupied spot in encounters, if any.
If car $2$ does not encounter any unoccupied spot, we say it is unable to park.
The parking process continues similarly for all subsequent cars.
If the preferences $\alpha=(a_1,\ldots,a_n)$ allow all cars to park, we say $\alpha$ is a \emph{parking function} of length $n$. 
Let $\PF_n$ denote the set of parking functions of length $n$.
For example, $(1,1,1,1)$ is a parking function of length four in which all cars prefer the first spot and, based on the parking rule, the cars 1, 2, 3, and 4 park in the order 1, 2, 3, and 4, respectively.

Parking functions were introduced by Konheim and Weiss in their study of hashing functions~\cite{konheim1966occupancy}.
They established that $|\PF_n|=(n+1)^{n-1}$.
Since their foundational result, parking functions have been modified and/or generalized in various ways.
We point the reader to Yan~\cite{yan2015parking} for a comprehensive survey of results.

In this work, we introduce a variant of parking functions in which later cars are considered to be a ``most valuable player'' (MVP).
We modify the ``classical parking rule'' (wherein a newly arrived car starts at its preferred spot and continues driving down the one-way road until it parks in the first unoccupied spot it encounters, if any) as follows.
Upon the arrival of car $i \in [n]$, it attempts to park in its preferred spot $a_i$.
If spot $a_i$ is unoccupied, it parks there.
Otherwise, if spot $a_i$ is occupied by an earlier-arriving car $j \in [n]$ with $j < i$, the MVP car $i$ ``bumps'' car $j$ out of spot $a_i$.
Then, the newly bumped car $j$ continues driving down the one-way street until it parks in the first unoccupied spot it encounters, if any.
We refer to this parking rule as the ``MVP parking rule.''
Note that the bumping of a car out of a spot happens only once\textemdash there is no cascading effect in which a newly bumped car itself bumps further cars while on its search for a new parking spot.

If the preferences $\alpha=(a_1,a_2,\ldots,a_n)\in[n]^n$ allow all cars to park under the MVP parking rule, we say $\alpha$ is an \emph{MVP parking function} of length $n$. 
Let $\MVP_n$ denote the set of MVP parking functions of length $n$. 
For example, $(1,1,1,1)$ is an MVP parking function of length four in which all cars prefer the first spot and, based on the MVP parking rule, the cars 1, 2, 3, and 4 park in the order 4, 1, 2, and 3, respectively. 

In Section~\ref{sec: preliminaries} we enumerate and characterize MVP parking functions by leveraging their connection to (classical) parking functions, and lay groundwork necessary to study their \emph{outcome map}.
Our first result (Theorem~\ref{theorem: v in PF iff v in MVP}) establishes that a list of preferences $\alpha \in [n]^n$ is an MVP parking function (of length $n$) if and only if it is a parking function (of length $n$), readily implying that $|\MVP_n|=(n+1)^{n-1}$. 
With this enumeration at hand, we focus on the outcomes of the parking processes (i.e.,~the order in which the cars park) and note that these can be vastly different depending on the parking rule.
To make this precise, let $\Sym_n$ denote the symmetric group on $n$ letters. 
Define the outcome map (under the classical parking rule) as $\OPFn:\PF_n\to\Sym_n$ given by $\OPFn(\alpha)=(\pi_1, \pi_2, \ldots, \pi_n)$, where car $\pi_i$ parks in spot $i$ given the preference list $\alpha$\textemdash note that this may or may not be the spot it preferred.
The outcome map under the classical parking rule appears in Stanley~\cite[Exercise 5.49(d,e)]{Stanley} and a detailed proof can be found in~\cite[Proposition 3.1]{knaplestats}.
The result states that if $\pi=(\pi_1,\pi_2,\ldots,\pi_n)\in\Sym_n$, then
\begin{align*}
    |\OPFn^{-1}(\pi)|\coloneqq|\{\alpha\in\PF_n\,:\, \OPFn(\alpha)=\pi\}|=\prod_{i=1}^n\ell(i;\pi),
\end{align*}
where $\ell(i;\pi)$ is the length of the longest subsequence $\pi_j,\ldots,\pi_i$ of $\pi$ such that $\pi_k\leq \pi_i$ for all $j\leq k \leq i$.  
The authors in \cite{knaplestats} extend this result to the outcome map under the ``$k$-Naples parking rule,'' which allows cars to back up to $k$ spots if they find their preferred spot occupied prior to continuing forward. 
They obtain an analogous result for the size of the fibers of the corresponding outcome map \cite[Theorem 3.1]{knaplestats}
and use it to give a nonrecursive formula for the number of $k$-Naples parking functions of length $n$ \cite[Theorem 3.1]{knaplestats}.

Motivated by the results in \cite{knaplestats}, in Section~\ref{sec:fibers} we study the fibers of the outcome map under the MVP parking rule, $\OMVPn:\MVP_n\to\Sym_n$ defined analogously by
\[\OMVPn(\alpha)=(\pi_1,\pi_2,\ldots,\pi_n),\]
where car $\pi_i$ parks in spot $i$ given the preference list $\alpha$, which again may or may not be the spot it preferred. 
Note that even though $\PF_n=\MVP_n$ as sets, for an $\alpha$ in these sets that is not a permutation, the outcomes $\OPFn(\alpha)$ and $\OMVPn(\alpha)$ can be vastly different. As we notice above, $\alpha=(1,1,1,1)$ results in $\Out_{{\text{PF}}_4}(\alpha)=(1,2,3,4)$ while $\Out_{{\text{MVP}}_4}(\alpha)=(4,3,2,1)$.

In Section~\ref{sec: the outcome map} we study the fibers of the MVP outcome map through the lens of permutation pattern avoidance.
Our first main result (Theorem~\ref{thm:upper bound}) gives an upper bound for the size of the fiber of $\pi\in \Sym_n$, namely the cardinality of the set
\begin{align*}
    \OMVPn^{-1}(\pi)\coloneqq\{v\in\MVP_n:\OMVPn(v)=\pi\}.
\end{align*}
We then give a complete characterization for when the bound in Theorem~\ref{thm:upper bound} is tight based on the associated permutation (determining the parking order) avoiding the permutation patterns 321 and 3412 (Theorem~\ref{thm:equality}). In Section~\ref{sec:special perms}, we consider special families of permutations and give the cardinality of their corresponding fibers. Among the permutations considered is the longest word $w_0=(n,n-1,\ldots,3,2,1)\in\Sym_n$, for which we establish that $|\OMVPn^{-1}(w_0)|$ is enumerated by the Motzkin numbers\footnote{OEIS \textcolor{blue}{\href{https://oeis.org/A001006}{A001006}}.
The Motzkin numbers first appearead in Motzkin~\cite{motzkin1948relations}. 
See Donaghey and Shapiro~\cite{donaghey1977motzkin} for a sample of the settings in which they arise.} (Theorem~\ref{thm:motzkin}). 
We also give families of permutations for which the cardinality of the corresponding fiber is  exponential and others in which it is linear. 
We also note that we provide a python implementation for MVP parking functions and the outcome map, which can be found in \cite{Code}.
We conclude the article with some open problems.

\section{Preliminaries on MVP parking functions}
\label{sec: preliminaries}

In this section, we enumerate and characterize MVP parking functions and provide some initial results setting up our analysis of the fibers of the outcome map.

\subsection{Enumerating and characterizing MVP parking functions}
\label{sec: enumerating and characterizing mvp parking functions}
We begin with our first result which establishes the set equality of $\PF_n$ and $\MVP_n$.
\begin{theorem}
\label{theorem: v in PF iff v in MVP}
Let $n \in \mathbb{N}$ and $\alpha \in [n]^n$.
Then, $\alpha \in \PF_n$ if and only if $\alpha \in \MVP_n$.
\end{theorem}
\begin{proof}
Note that for both the parking rule and the MVP parking rule, which spots are occupied after the arrival of a car is solely a function of which spots are occupied before its arrival\textemdash which cars occupy them is immaterial.
Therefore, for any $\alpha \in [n]^n$, it suffices to check that the set of spots occupied throughout the parking process (and in particular at its conclusion) is consistent among the two parking rules. 
To do this, for each parking rule we define a sequence of functions $f^1, f^2, \ldots, f^n:[n]^n\to\{0,1\}^n$ where for any $\alpha=(a_1,a_2,\ldots,a_n) \in [n]^n$ and any $i, j \in [n]$, the $j$th entry of $f^i(\alpha)$, denoted  $f_j^i(\alpha)$, 
is $1$ if spot $j$ is occupied after the arrival of car $i$ given $\alpha$, and $0$ otherwise.

Given $\alpha$, let $\chi(\alpha) = (\chi^1(\alpha), \chi^2(\alpha), \ldots, \chi^n(\alpha)) \in \{0,1\}^n$ be the outputs of such functions under the classical parking rule.
Similarly, given $\alpha$, let $\psi(\alpha) = (\psi^1(\alpha), \psi^2(\alpha), \ldots, \psi^n(\alpha)) \in \{0,1\}^n$ be the outputs of such functions under the MVP parking rule.
Next, we show that $\chi(\alpha) = \psi(\alpha)$.

When the first car arrives, it parks in its preferred spot under either rule.
Therefore, $\chi^1(\alpha) = \psi^1(\alpha)$.
Now, suppose by way of induction that $\chi^i(\alpha) = \psi^i(\alpha)$ for some $1 \leq i < n$.
We now need to show that $\chi^{i + 1}(\alpha) = \psi^{i+1}(\alpha)$.
Recall $a_{i+1} \in [n]$ is the preferred spot of car $i + 1$.
If $\chi_{a_{i+1}}^i(\alpha) = \psi_{a_{i+1}}^i(\alpha) = 0$, car $i+1$ parks in spot $a_{i+1}$ under either rule and so $\chi^{i+1}(\alpha) = \psi^{i+1}(\alpha)$ with $\chi_{a_{i+1}}^{i+1}(\alpha) = \psi_{a_{i+1}}^{i+1}(\alpha) = 1$.
If $\chi_{a_{i+1}}^i(\alpha) = \psi_{a_{i+1}}^i(\alpha) = 1$, there are two possible cases.
If there exists some $j > a_{i+1}$ such that $\chi_{j}^i(\alpha) = \psi_{j}^i(\alpha) = 0$, pick the smallest such $j$.
Under the classical parking rule, car $i + 1$ parks in spot $j$ and so $\chi_j^{i+1}(\alpha) = 1$ while $\chi_{a_{i+1}}^i(\alpha) = \chi_{a_{i+1}}^{i+1}(\alpha) = 1$ remains the same.
Under the MVP parking rule, car $i + 1$ parks in spot $a_{i+1}$ while the car that was parked in spot $a_{i+1}$ now parks in spot $j$.
Then, $\psi_{a_{i+1}}^i(\alpha) = \psi_{a_{i+1}}^{i+1}(\alpha) = 1$ remains the same while $\psi_j^{i+1}(\alpha) = 1$.
On the other hand, suppose there does not exist $j > a_{i+1}$ such that $\chi_{j}^i(\alpha) = \psi_{j}^i(\alpha) = 0$.
Under the classical parking rule, car $i + 1$ is unable to park and so $\chi^i(\alpha) = \chi^{i+1}(\alpha)$.
Under the MVP parking rule, car $i + 1$ parks in spot $a_{i+1}$ while the car that was parked in spot $a_{i+1}$ is now unable to park, and so $\psi^i(\alpha) = \psi^{i+1}(\alpha)$.
In both cases,
we have $\chi^{i+1}(\alpha) = \psi^{i+1}(\alpha)$.
\end{proof}

As a consequence of Theorem~\ref{theorem: v in PF iff v in MVP}, classical results on the enumeration and characterization of (classical) parking functions extend to MVP parking functions. This includes the following.
\begin{corollary}[\cite{konheim1966occupancy},~Lemma~1]
If $n\in\mathbb{N}$, then $|\MVP_n|=(n+1)^{n-1}$.
\end{corollary}

\begin{corollary}[See~\cite{yan2015parking}, pp. 836]
Let $n \in \mathbb{N}$ and $\alpha \in [n]^n$.
Let $b(\alpha) = (b_1(\alpha), b_2(\alpha), \ldots, b_n(\alpha))$ be the nondecreasing rearrangement of $\alpha$, so that $b_1(\alpha) \leq b_2(\alpha) \leq \cdots \leq b_n(\alpha)$.
Then, $\alpha \in \MVP_n$ if and only if $b_i(\alpha) \leq i$ for all $i \in [n]$.
\end{corollary}

Note that although $\PF_n=\MVP_n$, provided $\alpha\in\PF_n\setminus\Sym_n$, the outcome maps $\OPFn(\alpha)$ and 
$\OMVPn(\alpha)$ can be vastly different, and this is in no way explained by Theorem~\ref{theorem: v in PF iff v in MVP}.
Thus the remainder of the manuscript is dedicated to the question: Given a permutation $\pi\in \Sym_n$, can we characterize and enumerate the set of MVP parking functions $\alpha\in\MVP_n$ which satisfy $\OMVPn(\alpha)=\pi$?

\subsection{MVP parking functions and their outcome map}\label{sec:fibers}

For a finite set $S$, let $\Sym_S$ denote its set of permutations.
For ease of notation, we use $\Sym_n$ to denote $\Sym_{[n]}$.
For $\pi \in \Sym_n$, we adopt the following (unusual, but convenient one-line) notation $\pi = (\pi_1,\pi_2,\ldots,\pi_n)$ where, for $j \in [n]$, we denote $\pi_j = \pi(j)$. \begin{proposition}
For any $n \in \mathbb{N}$, $\OMVPn$ is a well-defined function.
\end{proposition}
\begin{proof}
Let $\alpha \in \MVP_n$.
By definition, after the arrival of cars with preferences in $\alpha$, each of the $n$ spots is occupied by one of the $n$ cars\textemdash we represent said configuration with $\pi \in \Sym_n$. 
That is, $\pi_j = i$ indicates spot $j$ is occupied by car $i$, where $i, j \in [n]$.
Note that $\pi$ is solely determined by $\alpha$ and the $\MVP$ parking rule.
Therefore, 
\begin{enumerate}[label=(\roman*)]
    \item $\OMVPn \subseteq \MVP_n \times \Sym_n$, where we think of $\OMVPn$ as a binary relation,
    \item for every $\alpha \in \MVP_n$, there exists $\pi \in \Sym_n$ such that $(\alpha, \pi) \in \OMVPn$, and 
    \item for every $\alpha \in \MVP_n$ and every $\pi, \pi' \in \Sym_n$, $(\alpha, \pi), (\alpha, \pi') \in \OMVPn$ implies $\pi = \pi'$.
\end{enumerate}
That is, $\OMVPn \subseteq \MVP_n \times \Sym_n$ is a total univalent relation\footnote{Recall that a function $f:X\to Y$ can be described as a binary relation on a subset $R \subseteq X \times Y$ that is univalent, i.e.~ $\forall x\in X$, 
$\forall y, z \in Y$, $((x,y)\in R\wedge (x,z)\in R)\rightarrow y=z$, 
and total, i.e.~$\forall x\in X, \exists y\in Y$, $(x,y)\in R$.}.
\end{proof}
For $\pi \in \Sym_n$, let
\begin{align*}
     \OMVPn^{-1}(\pi) = \{\alpha \in \MVP_n : \OMVPn(\alpha) = \pi \} \subseteq \MVP_n
\end{align*}
be the \emph{fiber} of $\pi$. 
As we show next, certain fibers are very simple.
\begin{proposition}
\label{prop: fiber of identity}
For any $n \in \mathbb{N}$, $\OMVPn^{-1}((1, 2, \ldots, n)) =\{(1,2, \ldots, n)\}$.
\end{proposition}
\begin{proof}
Clearly $(1, 2, \ldots, n) \in \OMVPn^{-1}((1, 2, \ldots, n))$.
To see the converse, suppose $\pi = (1, 2, \ldots, n)$ and let $\alpha = (a_1,a_2,\ldots,a_n) \in \OMVPn^{-1}((1, 2, \ldots, n))$. Recall that car $i' \in [n]$ bumps car $i \in [n]$ if and only if $i' > i$ and car $i$ occupies the spot preferred by car $i'$ upon the arrival of the latter.
Therefore, if car $1$ parks in spot $1$, it must be the case that car $1$ prefers spot $1$ and no subsequent car prefers spot $1$ (i.e., $\pi_1 = 1$ necessitates $a_1 = 1$ and $a_i > 1$ for all $i > 1$).
Similarly, if car $2$ parks in spot $2$, it must be the case that car $2$ prefers spot $2$ and no subsequent car prefers spot $2$ (i.e., $\pi_1 = 1$, $\pi_2 = 2$ necessitates $a_1 = 1$, $a_2 = 2$, and $a_i > 2$ for all $i > 2$).
We can extend this argument inductively to conclude that $\pi = (1, 2, \ldots, n)$ necessitates $\alpha = (1, 2, \ldots, n)$.
\end{proof}
Proposition~\ref{prop: fiber of identity} implies $|\OMVPn ^{-1}((1, 2, \ldots, n))| = 1$.
This is in fact as small as a fiber as can be\textemdash our next result implies that, for all $\pi \in \Sym_n$, we have $|\OMVPn^{-1}(\pi)| \geq 1$.
\begin{theorem}\label{thm:size>=1}
Let $\pi \in \Sym_n$. 
Then, $\OMVPn(\pi) = \pi^{-1}$.
\end{theorem}
\begin{proof}
% To begin we remark that since $\pi$ is a permutation the result $\OMVPn(\pi)$ is as well. 
It suffices to show that $\pi \in \OMVPn^{-1}(\pi^{-1})$.
Consider the preference vector $\alpha = (a_1, a_2, \ldots, a_n)$ where, if $\pi_{i}^{-1} = j$, then $a_j = i$.
We have $\alpha \in \OMVPn^{-1}(\pi^{-1})$ by construction.
We claim $\alpha \in \Sym_n$.
To prove this, it suffices to show that for every $i \in [n]$, there exists a unique $j \in [n]$ satisfying $a_j = i$.
By way of contradiction, suppose there exists $i \in [n]$ such that there does not exist a unique $j \in [n]$ satisfying $a_j = i$, and fix any such $i$.
If there is no $j \in [n]$ satisfying $a_j = i$, then $\pi_{i}^{-1} \notin [n]$, contradicting $\pi^{-1} \in \Sym_n$.
Similarly, if there are distinct $j, j' \in [n]$ satisfying $a_j = a_{j'} = i$, then $\pi_i^{-1} = j$ and $\pi_i^{-1} = j'$, contradicting $\pi^{-1} \in \Sym_n$.
Therefore, it remains to show that $\alpha = \pi$, which is to show $\alpha \cdot \pi^{-1} = (1, 2, \ldots, n)$.
Let $i \in [n]$. 
Then, $\pi_i^{-1} = j$ for some $j \in [n]$, in which case $a_j = i$ by construction.
That is, $(\alpha \cdot \pi^{-1})_i = i$ for all $i \in [n]$, implying $\alpha = \pi$.
\end{proof}
\begin{corollary}
For each $\pi \in \Sym_n$, we have $|\OMVPn^{-1}(\pi)| \geq 1$.
\end{corollary}
\begin{proof}
For each $\pi \in \Sym_n$, there exists a unique $\pi^{-1} \in \Sym_n$ such that $\pi \cdot \pi^{-1} = (1, 2, \ldots, n)$.
\end{proof}
Before stating our next result we recall that a permutation $\pi$ is said to be an involution if $\pi=\pi^{-1}$. Moreover, $\pi$ is an involution if consists exclusively of fixed points and disjoint transpositions.
\begin{corollary}
If $\pi \in \Sym_n$ is an involution, then $\OMVPn(\pi) = \pi$.
\end{corollary}
\begin{proof}
If $\pi \in \Sym_n$ is an involution, then $\pi \cdot \pi = (1, 2, \ldots, n)$ (i.e., $\pi = \pi^{-1}$).
\end{proof}

\section{The outcome map and permutations avoiding $(3,2,1)$ and $(3,4,1,2)$}
\label{sec: the outcome map}

Next we give an upper bound on the cardinality of the fibers of the outcome map. This work relies on the following definitions.
\begin{definition}\label{def:cars}
Let $\pi \in \Sym_n$ where $\pi_j = i$ indicates the $j$th spot is occupied by the $i$th car.
For each $j \in [n]$, find $i \in [n]$ such that $\pi_j = i$. 
Then, let
\begin{align*}
    \mathcal{C}_j(\pi) \coloneqq \left( \{\pi_1,\pi_2,\ldots,\pi_{j-1} \} \cap \{i + 1, i + 2, \ldots, n\} \right) \cup \{i\}
\end{align*}
be the set of cars that arrive \emph{after} the $i$th car (i.e., the cars numbered $i + 1, i + 2, \ldots, n$) that park to the \emph{left} of spot $j$ (i.e., in spots $1, 2, \ldots, j-1$), together with the $i$th car itself.
\end{definition}
\begin{definition}\label{prefSpots}
Let $\pi \in \Sym_n$ where $\pi_j = i$ indicates the $j$th spot is occupied by the $i$th car.
For each $j \in [n]$, let 
\begin{align*}
    \Omega_j(\pi) \coloneqq \{k \in [j] : \pi_k \in \mathcal{C}_j(\pi)\}
\end{align*}
be the set of spots on or to the left of the $j$th spot (i.e., the spots numbered $1, 2, \ldots, j$) that have a car in the set $\mathcal{C}_j(\pi)$ parked in them.
Note that necessarily $j \in \Omega_j(\pi)$.
\end{definition}

We remark that Definition~\ref{def:cars} is similar to a Lehmer code for permutations, which is defined as $
L(\sigma )=(L(\sigma )_{1},\ldots ,L(\sigma )_{n})\quad {\text{where}}\quad L(\sigma )_{i}=\#\{j>i:\sigma _{j}<\sigma _{i}\},$
i.e., $L(\sigma_i)$ counts the number of terms in $(\sigma_1,\ldots,\sigma_n)$ to the right of $\sigma_i$ that are smaller than it. For more on Lehmer codes we point the interested reader to \cite{Lehmer}.

\begin{example}
Let $\pi = 341526$. Then 
$
    \mathcal{C}_2(\pi) \coloneqq \{3,4,5,2\}
$
since cars $3,4,5$ arrived after car $2$ and have parked left of car $2$. Then $
    \Omega_2(\pi) \coloneqq \{1,2,4,5\}
$,
which correspond to the spots cars $3,4,5,2$ occupy in $\pi$.
\end{example}

With these definitions at hand, we obtain the following upper bound.
\begin{theorem}\label{thm:upper bound}
If $\pi \in \Sym_n$,
then
\begin{align}\label{main:inequality}
    |\OMVPn^{-1}(\pi)| \leq \prod_{j=1}^{n}|\Omega_{j}(\pi)|.
\end{align}
\end{theorem}
\begin{proof}
Let $\alpha = (a_1,a_2,\ldots,a_n) \in \OMVPn^{-1}(\pi)$. Note that for each $i,j \in [n]$ with $\pi_j = i$, it must be the case that the preference $a_i$ of the $i$th car satisfies $a_i \in \Omega_j(\pi)$.
To see this, note that by the MVP parking rule, if car $i' \in [n]$ with $i' > i$ ultimately parks to the left of the $j$th spot, it may have bumped car $i$ (this holds if car $i$ occupies the spot preferred by car $i'$ upon the arrival of the latter) to spot $j$.
\end{proof}

The bound in Theorem~\ref{thm:upper bound} is not tight in general since, for example, if $\pi = (1,4,6,5,2,3) \in \Sym_6$, then we have computed\footnote{Code for these computation can be found in \cite{Code}.} that $|\OMVPn^{-1}(\pi)| = 13 < 32 = \prod_{j=1}^6 |\Omega_j(\pi)|$. 
However, there are $\pi\in\Sym_n$ for which the bound in~\eqref{main:inequality} is in fact an equality. In such cases, we say that $\pi$ achieves \emph{preference independence}.
We use this wording to emphasize that one car's parking preference do not impact the preference of other cars.

\begin{example}
Consider the outcome $\pi = (5,1,2,3,6,9,4,7,8)$. We can check that $|\OMVPn^{-1}(\pi)| = 128 = \prod_{j=1}^9 |\Omega_j(\pi)|$, so $\pi$ achieves preference independence.
\end{example}

In the remainder of this section, we characterize the conditions under which $\pi \in \Sym_n$ achieves preference independence. 
To do so, we need the following definitions.

\begin{definition}
Let $(S, \leq_S)$ and $(T, \leq_T)$ be totally ordered sets on $m \in \mathbb{N}$ elements.
Let $s \in \Sym_S$ and $t \in \Sym_T$.
We say $s$ and $t$ are order-isomorphic if, for all $i, j \in [m]$ with $i < j$, $s_i \leq s_j$ if and only if $t_i \leq t_j$.
We denote order-isomorphism by $s \sim t$.
\end{definition}
For example, let $S = \{1, 2, 3, 4\}$ and $T = \{3, 4, 6, 9\}$, both with the standard ordering.
Consider $s = (2, 3, 1, 4) \in \Sym_S$ and $t = (4, 6, 3, 9) \in \Sym_T$, and note that $s \sim t$.
\begin{definition}
Let $\pi \in \Sym_n$ and $M \subseteq [n]$.
We say $\pi$ contains $\rho \in \Sym_{M}$ if $\pi$ has a subpermutation $\pi'$ such that $\pi' \sim \rho$.
Otherwise, we say $\pi$ avoids $\rho$.
\end{definition}

\begin{example}
Note $\pi = (7,\underline{3},6,2,\underline{5},4,\underline{1})$ contains $(2,3,1)$, since the subpermutation $(3,5,1) \sim (2,3,1$). Also 
$\pi = (7,\underline{3},6,2,\underline{5},\underline{4},1)$ contains $(1,3,2)$, since the subpermutation $(3,5,4) \sim (1,3,2)$.
However, $\pi = (7,3,6,2,5,4,1)$ avoids $(1,2,3)$ as there is no subpermutation in $\pi$ that is order-isomorphic to $(1,2,3)$.
\end{example}

With these definitions at hand we are ready to return to our question of interest: For what permutations does Theorem~\ref{thm:upper bound} result in preference independence and, hence, an equality in \eqref{main:inequality}? The following results fully establish this characterization. 

\begin{proposition}\label{inequality}
Let $\pi \in \Sym_n$. 
If $\pi$  contains $(3,2,1)$ or $(3,4,1,2)$, then
\[|\OMVPn^{-1}(\pi)|<\prod_{j=1}^{n}|\Omega_{j}(\pi)|.\]
\end{proposition}

\begin{proof}
Let $\pi \in \Sym_n$, where $\pi_j = i$ indicates the $j$th spot is occupied by the $i$th car.

First, suppose $\pi$ contains $(3,2,1)$.
Then, there exist $x, y, z \in [n]$ satisfying $x < y< z$ and $\pi_x > \pi_y > \pi_z$. 
By Definition~\ref{prefSpots} we have $\{x,y,z \} \subseteq \Omega_{\pi_z} (\pi)$, $\{x,y\} \subseteq \Omega_{\pi_y} (\pi)$, and $\{x\} \subseteq \Omega_{\pi_x}(\pi)$.
Suppose car $\pi_x$ prefers spot $x$ and car $\pi_y$ prefers spot $x$.
Suppose moreover that car $\pi_i$ prefers spot $i$ for all $i \in [n]$ with $i \neq x, y, z$ (recall Definition~\ref{prefSpots} implies $\{i\} \subseteq \Omega_{i}(\pi)$ for all $i \in [n]$). 
We claim car $\pi_z$ cannot prefer spot~$y$. 
Assume by way of contradiction that car $\pi_z$ prefers spot $y$. 
Since $\pi_x > \pi_y > \pi_z$, car $\pi_z$ arrives first and parks in spot $y$.
Car $\pi_y$ arrives later and parks in spot $x$.
Upon the arrival of car $\pi_x$, car $\pi_x$ parks in spot $x$ and bumps car $\pi_y$.
Since car $\pi_i$ prefers spot $i$ for all $i \in [n]$ with $i \neq x,y,z$ and since since spot $y$ is already occupied by some car $\pi_u$ with $u \geq z$ and $u \neq y$, ultimately car $\pi_y$ is bumped to the right of spot $y$, a contradiction.

Next, suppose $\pi$ contains $(3,4,1,2)$.
Then, there exist $w,x,y,z \in [n]$ satisfying $w < x < y < z$, $\pi_x > \pi_w > \pi_z > \pi_y$, and $\pi_v < \pi_x$ for any $w < v < y$ with $v \neq x$\textemdash the last condition is to say that car $\pi_x$ is the latest-arriving car that parks between spots $w$ and $y$. 
By Definition~\ref{prefSpots} we have $\{w,x,z\} \subseteq \Omega_{\pi_z} (\pi)$, $\{w, x, y\} \subseteq \Omega_{\pi_y}(\pi)$, $\{x\} \subseteq \Omega_{\pi_x}(\pi)$, and $\{w\} \subseteq \Omega_{\pi_w}(\pi)$.
Suppose car $\pi_y$ prefers spot $x$, car $\pi_w$ prefers spot $w$, and car $\pi_x$ prefers spot $x$.
Suppose moreover that car $\pi_i$ prefers spot $i$ for all $i \in [n]$ with $i \neq w, x, y, z$ (recall Definition~\ref{prefSpots} implies $\{i\} \subseteq \Omega_{i}(\pi)$ for all $i \in [n]$). 
We claim car $\pi_z$ cannot prefer spot $w$.
Assume by way of contradiction that car $\pi_z$ prefers spot $w$.
Since $\pi_x > \pi_w > \pi_z > \pi_y$, car $\pi_y$ arrives first and parks in spot $x$.
Car $\pi_z$ arrives later and parks in spot $w$. 
Car $\pi_w$ arrives later and parks in spot $w$, bumping car $\pi_z$.
Note that, by the time car $\pi_x$ arrives, car $\pi_z$ has been bumped to spot $y$ since car $\pi_i$ prefers spot $i$ for all $i \in [n]$ with $i \neq w,x,y,z$ and since car $\pi_x$ is the latest-arriving car that parks between spots $w$ and $y$. 
Therefore, upon the arrival of car $\pi_x$, car $\pi_x$ parks in spot $x$ and bumps car $\pi_y$ to the right of the already occupied spot $y$, a contradiction.
\end{proof}

Given $\pi \in \Sym_n$ and an arbitrary preference vector $\alpha \in [n]^n$, we note that a car of $\alpha$ can always park where it appears in $\pi$, by making the ``trivial preference.'' This is the rough idea of the following lemma.
\begin{lemma}
\label{trivial pref}
Let $\pi = (\pi_1,\pi_2,\ldots,\pi_n) \in \Sym_n$, and let $\alpha \in [n]^n$ be a preference vector such that car $\pi_j$ can only prefer spots in $\Omega_j(\pi)$, for all $j \in [n]$. Suppose car $\pi_i$ prefers spot $i$, for some $i \in [n]$. Then car $\pi_i$ ultimately parks in spot $i$.
\end{lemma}
\begin{proof}
By Definition~\ref{prefSpots}, any car $\pi_u > \pi_i$ cannot prefer spot $i$. If any car $\pi_v < \pi_i$ parked in spot $i$, then car $\pi_i$ bumps car $\pi_v$ to another spot further down the street. Hence car $\pi_i$ parks in spot~$i$.
\end{proof}

Now we prove the converse of Proposition~\ref{inequality}, which we state using the contrapositive as follows. 

\begin{proposition}\label{equality}
If $\pi \in \Sym_n$ avoids $(3,2,1)$ and $(3,4,1,2)$, then $|\OMVPn^{-1}(\pi)|=\prod_{j=1}^{n}|\Omega_{j}(\pi)|.$
\end{proposition}
\begin{proof}
Let $\pi = (\pi_1,\pi_2, \ldots,\pi_n)$. It suffices to show that the preference of one car does not influence how another car parks in $\pi$. 
This way, we demonstrate that $\pi$ satisfies \textit{preference independence}. 
We proceed by induction on $k$, the number of cars $\pi_1,\pi_2 ,\ldots, \pi_k$ starting from the left of $\pi$.

In the base case $k=1$, note that car $\pi_1$ can and must prefer spot 1 in order to park there. 
Now assume that up to some $k < n$, the cars $\pi_1,\pi_2,\pi_3, \ldots, \pi_k$ can all prefer independently.
We show that car $\pi_{k+1}$ also prefers independently. We consider the following cases:
\begin{enumerate}[label=\roman*.]
\setlength{\itemindent}{1em}
    \item[\textbf{Case 1:}] There does not exist $j \leq k $ such that $\pi_j > \pi_{k+1} $. 
    Hence the cars parked to the left of car $\pi_{k+1}$, which is parked in spot $k+1$, were in the queue before it, i.e.~$\pi_i < \pi_{k+1}$ for all  $1 \leq i \leq k$. 
     By Definition~\ref{prefSpots}, then, $\pi_{k+1}$ only prefers spot $k+1$. It follows that car $\pi_{k+1}$ parks in spot $k+1$ as desired, by Lemma \ref{trivial pref}. 
     \item[\textbf{Case 2:}]
     There exists exactly one $j \leq k $ such that $\pi_j > \pi_{k+1} $.
     In this case, for every $i\in[n]$ we plot $(i,\pi_i)$ on the lattice $[n]\times[n]$ and note that the structure of the permutation $\pi$ is illustrated in Figure~\ref{fig:case2}.
     \begin{figure}[H]
         \centering
\resizebox{3in}{!}{
         \tikzset{every picture/.style={line width=0.75pt}} %set default line width to 0.75pt        

\begin{tikzpicture}[x=0.75pt,y=0.75pt,yscale=-1,xscale=1]
%uncomment if require: \path (0,415); %set diagram left start at 0, and has height of 415

%Shape: Axis 2D [id:dp2578736461960899] 
\draw  (49,301.2) -- (442,301.2)(88.3,6) -- (88.3,334) (435,296.2) -- (442,301.2) -- (435,306.2) (83.3,13) -- (88.3,6) -- (93.3,13)  ;
%Straight Lines [id:da15236113222947034] 
\draw    (249.5,27) -- (249,302) ;
%Flowchart: Connector [id:dp0787540708315323] 
\draw  [fill={rgb, 255:red, 0; green, 0; blue, 0 }  ,fill opacity=1 ] (332,160) .. controls (332,157.24) and (334.24,155) .. (337,155) .. controls (339.76,155) and (342,157.24) .. (342,160) .. controls (342,162.76) and (339.76,165) .. (337,165) .. controls (334.24,165) and (332,162.76) .. (332,160) -- cycle ;
%Flowchart: Connector [id:dp843217874275228] 
\draw  [fill={rgb, 255:red, 0; green, 0; blue, 0 }  ,fill opacity=1 ] (244,76) .. controls (244,73.24) and (246.24,71) .. (249,71) .. controls (251.76,71) and (254,73.24) .. (254,76) .. controls (254,78.76) and (251.76,81) .. (249,81) .. controls (246.24,81) and (244,78.76) .. (244,76) -- cycle ;
%Shape: Rectangle [id:dp3602598271841273] 
\draw  [fill={rgb, 255:red, 0; green, 0; blue, 0 }  ,fill opacity=0.2 ] (88,27) -- (337,27) -- (337,159) -- (88,159) -- cycle ;
%Shape: Rectangle [id:dp6267552155515528] 
\draw  [fill={rgb, 255:red, 0; green, 0; blue, 0 }  ,fill opacity=0.2 ] (339,160) -- (418,160) -- (418,301) -- (339,301) -- cycle ;
%Shape: Rectangle [id:dp6093124529092961] 
\draw   (337,28) -- (418,28) -- (418,160) -- (337,160) -- cycle ;

% Text Node
\draw (99,309) node [anchor=north west][inner sep=0.75pt]    {${1\ \ \ \ 2\ \ \ \ 3\ \ \ \ \ \ \ \cdots \ \ \ \ \ \  j\ \ \ \ \ \ \cdots \ \ \ \ k+1\ \ \ \ \cdots \ \ \ n\ }$};
% Text Node
\draw (150,223) node [anchor=north west][inner sep=0.75pt]    {Region\ $ 3$};
% Text Node
\draw (150,90) node [anchor=north west][inner sep=0.75pt]    {Region\ $1$};
% Text Node
\draw (350,223) node [anchor=north west][inner sep=0.75pt]    {Region\ $ 2$};
% Text Node
\draw (350,90) node [anchor=north west][inner sep=0.75pt]    {Region\ $ 4$};
% Text Node
\draw (255,47) node [anchor=north west][inner sep=0.75pt]    {$\pi _{j}$};
% Text Node
\draw (341,134) node [anchor=north west][inner sep=0.75pt]    {$\pi _{k+1}$};
% Text Node
\draw (40,18) node [anchor=north west][inner sep=0.75pt]    {$ \begin{array}{l}
\ \ \ \ \ n\\
\ \ \ \ \\
\\
\\
\\
\\
\ \ \ \ \vdots \\
\\
\\
\\
\\
\ \ \ \ \\
\\\\
\ \ \ \ \ 1
\end{array}$};

\end{tikzpicture}
         }
         \caption{Illustrating $\pi$  in Case 2. Note that the shaded regions contains no points $(i,\pi)$ other than the ones included in the graphic.
         }
         \label{fig:case2}
     \end{figure}
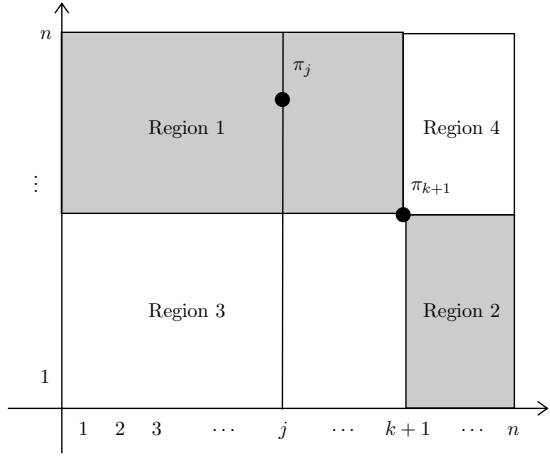
     We begin by first noting that by Definition~\ref{prefSpots}, car $\pi_{k+1}$ prefers spots $\{j,k+1\}$. We will show that car $\pi_{k+1}$ can prefer these spots independently.
     
     First we remark that, by assumption in this case, there exists only one index $j$ where $\pi_j > \pi_{k+1}$. 
     This implies that cars $\pi_1,\ldots \pi_{j-1},\pi_{j+1}, \ldots, \pi_k < \pi_{k+1}$, park in spots $\{1,\ldots, k\}\setminus\{j\}$ and hence the corresponding points $(i,\pi_i)$ for $i\in\{1,\ldots,k\}\setminus\{j\}$ all lie in \texttt{Region 3} of Figure~\ref{fig:case2}. This implies that \texttt{Region 1} in Figure~\ref{fig:case2} contains the point $(j,\pi_j)$ and is empty otherwise.
     Next, since $\pi$ avoids the pattern $(3,2,1)$,  cars $\pi_{k+2}, \ldots, \pi_n > \pi_{k+1}$ all must park  in spots $\{k+2,\ldots, n\}$ and hence the corresponding points $(i,\pi_i)$ for $i\in\{k+2,\ldots,n\}$ all lie in \texttt{Region 4} of Figure~\ref{fig:case2}. This implies that \texttt{Region 2} of Figure~\ref{fig:case2} contains  the point $(k+1,\pi_{k+1})$ and is empty otherwise.
     
     Now suppose car $\pi_{k+1}$ preferred spot $j$.  By our inductive hypothesis, we know that cars $\pi_1, \pi_2,\ldots, \pi_{j-1}, \pi_{j+1}, \ldots, \pi_k$ all park independently between spots $1,\ldots,k$. 
     When car $\pi_{k+1}$ enters the street, it parks in spot $j$, bumping any other car that may have parked there to a spot $m < k+1$. 
     Then car $\pi_{k+1}$ gets bumped by car $\pi_j$ or car $\pi_{j'}$ afterwards to spot $k+1$, where $\pi_{k+1} < \pi_{j'} < \pi_j$. We know that there are no empty spots between $j$ and $k+1$ since by the inductive hypothesis, all the cars $\pi_1, \ldots, \pi_k$ have now already parked independently within those spots. 
     Therefore car $\pi_{k+1}$ will park in the first available spot, which is spot $k+1$. 
     Note that car $\pi_{k+1}$ will not get bumped out of this spot again since $\pi_{k+2}, \ldots, \pi_n > \pi_{k+1}$ and by Definition~\ref{prefSpots}, those cars will never prefer spot $k+1$.
     
     On the other hand, suppose car $\pi_{k+1}$ preferred spot $k+1$. Then car $\pi_{k+1}$ parks in spot $k+1$ as desired, by Lemma \ref{trivial pref}.
     
     We have therefore shown that car $\pi_{k+1}$, given all of its preferences, parks independently under the assumptions of Case 2.

  \item[\textbf{Case 3:}]
     There exist $m\geq 2$ and indices $j_1, j_2, \ldots, j_m$,  where $j_1 < j_2 < \cdots < j_m < k+1$, such that $\pi_{j_i} > \pi_{k+1}$ for all $i\in[m]$.    
     
     First, notice that if the cars $\pi_{j_1}, \ldots, \pi_{j_m}$ are not in increasing order, then that implies that there exists two indices $j_a,j_b$ where $1 < a < b \leq m$ 
     and $\pi_{j_a} > \pi_{j_b}$. It follows then that there exists a $(3,2,1)$ pattern, giving rise to a contradiction. 
     Figure~\ref{noOrder} illustrates this case.\newline

     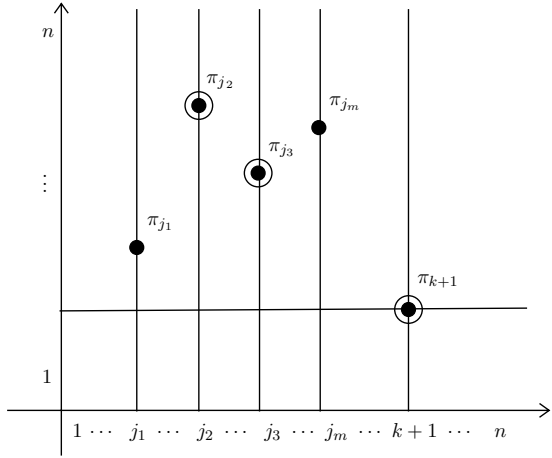
\begin{figure}[H]
         \centering
\resizebox{3in}{!}{

\tikzset{every picture/.style={line width=0.75pt}} %set default line width to 0.75pt        

\begin{tikzpicture}[x=0.75pt,y=0.75pt,yscale=-1,xscale=1]
%uncomment if require: \path (0,415); %set diagram left start at 0, and has height of 415

%Shape: Axis 2D [id:dp5045254239655744] 
\draw  (49,301.2) -- (442,301.2)(88.3,6) -- (88.3,334) (435,296.2) -- (442,301.2) -- (435,306.2) (83.3,13) -- (88.3,6) -- (93.3,13)  ;
%Flowchart: Connector [id:dp39921839884182886] 
\draw  [fill={rgb, 255:red, 0; green, 0; blue, 0 }  ,fill opacity=1 ] (335,228) .. controls (335,225.24) and (337.24,223) .. (340,223) .. controls (342.76,223) and (345,225.24) .. (345,228) .. controls (345,230.76) and (342.76,233) .. (340,233) .. controls (337.24,233) and (335,230.76) .. (335,228) -- cycle ;
%Flowchart: Connector [id:dp718697674950395] 
\draw  [fill={rgb, 255:red, 0; green, 0; blue, 0 }  ,fill opacity=1 ] (138,183) .. controls (138,180.24) and (140.24,178) .. (143,178) .. controls (145.76,178) and (148,180.24) .. (148,183) .. controls (148,185.76) and (145.76,188) .. (143,188) .. controls (140.24,188) and (138,185.76) .. (138,183) -- cycle ;
%Straight Lines [id:da6601590493773836] 
\draw    (143,10) -- (143,302) ;
%Straight Lines [id:da988684715052803] 
\draw    (188,10) -- (188,302) ;
%Straight Lines [id:da4217646418335982] 
\draw    (232,10) -- (232,302) ;
%Straight Lines [id:da0416115734533794] 
\draw    (276,10) -- (276,302) ;
%Flowchart: Connector [id:dp8675131193177048] 
\draw  [fill={rgb, 255:red, 0; green, 0; blue, 0 }  ,fill opacity=1 ] (183,80) .. controls (183,77.24) and (185.24,75) .. (188,75) .. controls (190.76,75) and (193,77.24) .. (193,80) .. controls (193,82.76) and (190.76,85) .. (188,85) .. controls (185.24,85) and (183,82.76) .. (183,80) -- cycle ;
%Flowchart: Connector [id:dp7436544424413184] 
\draw  [fill={rgb, 255:red, 0; green, 0; blue, 0 }  ,fill opacity=1 ] (226,129) .. controls (226,126.24) and (228.24,124) .. (231,124) .. controls (233.76,124) and (236,126.24) .. (236,129) .. controls (236,131.76) and (233.76,134) .. (231,134) .. controls (228.24,134) and (226,131.76) .. (226,129) -- cycle ;
%Flowchart: Connector [id:dp8508049195039643] 
\draw  [fill={rgb, 255:red, 0; green, 0; blue, 0 }  ,fill opacity=1 ] (270,96) .. controls (270,93.24) and (272.24,91) .. (275,91) .. controls (277.76,91) and (280,93.24) .. (280,96) .. controls (280,98.76) and (277.76,101) .. (275,101) .. controls (272.24,101) and (270,98.76) .. (270,96) -- cycle ;
%Straight Lines [id:da969039485179241] 
\draw    (87,229) -- (426,227) ;
%Shape: Circle [id:dp10480953355179012] 
\draw   (178,80) .. controls (178,74.48) and (182.48,70) .. (188,70) .. controls (193.52,70) and (198,74.48) .. (198,80) .. controls (198,85.52) and (193.52,90) .. (188,90) .. controls (182.48,90) and (178,85.52) .. (178,80) -- cycle ;
%Shape: Circle [id:dp9859353514920259] 
\draw   (221,129) .. controls (221,123.48) and (225.48,119) .. (231,119) .. controls (236.52,119) and (241,123.48) .. (241,129) .. controls (241,134.52) and (236.52,139) .. (231,139) .. controls (225.48,139) and (221,134.52) .. (221,129) -- cycle ;
%Shape: Circle [id:dp31085417692591166] 
\draw   (330,228) .. controls (330,222.48) and (334.48,218) .. (340,218) .. controls (345.52,218) and (350,222.48) .. (350,228) .. controls (350,233.52) and (345.52,238) .. (340,238) .. controls (334.48,238) and (330,233.52) .. (330,228) -- cycle ;
%Straight Lines [id:da9962239272179586] 
\draw    (340,10) -- (340,302) ;

% Text Node
\draw (95,309) node [anchor=north west][inner sep=0.75pt]    {$1\ {\cdots \ \ j_{1}\ \cdots \ \ j_{2}\ \cdots\ \ j_{3} \ \cdots\ j_{m} \ \cdots\  k+1\ \cdots \ \ \ n\ }$};
% Text Node
\draw (149,159) node [anchor=north west][inner sep=0.75pt]    {$\pi _{j_{1}}$};
% Text Node
\draw (345,201) node [anchor=north west][inner sep=0.75pt]    {$\pi _{k+1}$};
% Text Node
\draw (42,17) node [anchor=north west][inner sep=0.75pt]    {$ \begin{array}{l}
\ \ \ \ \ n\\
% \ \ \ \ \ \tiny{\vdots }\\
%  \ \ \ \pi_{j_2}

% \\
\\
\\
\\
 
 \\
 
 \\
 \ \ \ \ \ {\vdots}\\\\
 \\
 \\
 \\\\
 \\\\
%  \ \ \ \pi_{j_m}

% \\
% \ \ \ \ \ \tiny{\vdots} \\
%  \ \ \ \pi_{j_3}

% \\
% \ \ \ \ \ \tiny{\vdots }\\
%  \ \ \ \pi_{j_1}
% \\
% \ \ \ \ \ \tiny{\vdots}
% \\\pi_{k+1}
% \\
% \ \ \ \ \ \tiny{\vdots}
% \\\\
\ \ \ \ \ 1
\end{array}$};
% Text Node
\draw (192,56) node [anchor=north west][inner sep=0.75pt]    {$\pi _{j}{}_{_{2}}$};
% Text Node
\draw (237,105) node [anchor=north west][inner sep=0.75pt]    {$\pi _{j_{3}}$};
% Text Node
\draw (281,72) node [anchor=north west][inner sep=0.75pt]    {$\pi _{j_{m}}$};

\end{tikzpicture}
         }
         \caption{Image showing the case where the cars $\pi_{j_1}, \ldots, \pi_{j_m}$ are not in increasing order. The three circled points illustrate a sample $(3,2,1)$ pattern arising in $\pi$. 
        Notice that any such instance of these points not in increasing order would mean there is a $(3,2,1)$ pattern in $\pi$.
         }
         \label{noOrder}
     \end{figure}
    
     Therefore, since $\pi$ avoids $(3,2,1)$ we have established that $\pi_{j_1}, \ldots, \pi_{j_m}$ are in increasing order. 
     Also $\pi$ must avoid the pattern $(3,4,1,2)$. 
     For this to be true it must be that the cars $\pi_{k+2}, \ldots, \pi_n > \pi_{j_{m-1}}$, and hence the points $(a,\pi_a)$, with $k+2\leq a\leq n$, all lie in $\mathtt{Region\,3}$ of Figure~\ref{incOrder}.

We now consider how car $\pi_{k+1}$ parks in this case. By Definition~\ref{prefSpots}, we know that car $\pi_{k+1}$ prefers spots $\{j_1,j_2,\ldots, j_m,k+1\}$. 
Suppose car $\pi_{k+1}$ prefers spot $j_i$ for some $i\in[m]$. 
Since car $\pi_{j_i}$ must park in spot $k+1$, car $\pi_{k+1}$ must be bumped from spot $j_i$ by car $\pi_{j_i}$ or some car $\pi_{j'_i}$ satisfying $\pi_{k+1} < \pi_{j'_i} < \pi_{j_i}$. 
Once bumped, car $\pi_{k+1}$ cannot park in spot $u$ if $\pi_u < \pi_{k+1}$, otherwise car $\pi_u$ will park incorrectly due to spot $u$ being occupied by a car arriving after it, contradicting the inductive hypothesis. 
Thus, after being bumped, car $\pi_{k+1}$ must park in some spot $j_t$ where $ i < t \leq m$, where it will then be bumped by car $\pi_{j_t}$ or by some car $\pi_{j'_t}$ satisfying $\pi_{k+1} < \pi_{j'_t} < \pi_{j_t}$. 
This process is repeated until car $\pi_{k+1}$ is bumped past spot $j_m$. It remains to show that no car other than $\pi_{k+1}$ parks in spot $k+1$.

By the inductive hypothesis, none of the cars $\pi_1, \pi_2, \ldots, \pi_k$ can park in spot $k+1$. 
Moreover, the cars $\pi_{k+2}, \pi_{k+3}, \ldots, \pi_n$ are all in $\mathtt{Region\,3}$ of Figure~\ref{incOrder}, so they cannot prefer spot $k+1$, and $j_m$ is the only spot they can prefer which is left of $k+1$. To show that these cars cannot park in spot $k+1$ either, suppose that some of them prefer spot $j_m$. For $a_1, a_2, \ldots, a_u \in \{k+2,k+3,\ldots,n\}$, let $\pi_{a_1} < \pi_{a_2} < \ldots < \pi_{a_u}$ be the cars preferring spot $j_m$. 
By the time car $\pi_{a_1}$ enters the parking lot, car $\pi_{k+1}$ will have already been bumped past spot $j_{m-1}$, since $\pi_{a_1} > \pi_{j_{m-1}}$. 
If car $\pi_{k+1}$ parked in spot $k+1$, it parks there and we are done. 
If car $\pi_{k+1}$ parked in spot $j_m$, then it gets bumped by car $\pi_{a_1}$ and again parks in spot $k+1$; all the cars $\pi_{a_l}$ will get bumped past spot $k+1$. Hence none of the cars $\pi_{a_i}$ park in spot $k+1$ for all $i\in[u]$, as claimed. As these were the only possible cars that could potentially park in spot $k+1$, we have established that no car other than $\pi_{k+1}$ parks in spot $k+1$, as desired.

We again note that car $\pi_{k+1}$ preferring spot $k+1$ is the trivial case. Thus we conclude that car $\pi_{k+1}$ prefers independently under this case. Thereby completing the proof of this case.
     \begin{figure}[H]
         \centering
\resizebox{3in}{!}{

\tikzset{every picture/.style={line width=0.75pt}} %set default line width to 0.75pt        

\begin{tikzpicture}[x=0.75pt,y=0.75pt,yscale=-1,xscale=1]
%uncomment if require: \path (0,415); %set diagram left start at 0, and has height of 415

%Shape: Rectangle [id:dp2601518409580297] 
\draw  [fill={rgb, 255:red, 190; green, 185; blue, 185 }  ,fill opacity=1 ] (354,115) -- (426,115) -- (426,227) -- (354,227) -- cycle ;
%Straight Lines [id:da34845999809351247] 
\draw    (354,10) -- (354,302) ;
%Shape: Axis 2D [id:dp10137445025605485] 
\draw  (49,301.2) -- (442,301.2)(88.3,6) -- (88.3,334) (435,296.2) -- (442,301.2) -- (435,306.2) (83.3,13) -- (88.3,6) -- (93.3,13)  ;
%Flowchart: Connector [id:dp9519914781713443] 
\draw  [fill={rgb, 255:red, 0; green, 0; blue, 0 }  ,fill opacity=1 ] (349,228) .. controls (349,225.24) and (351.24,223) .. (354,223) .. controls (356.76,223) and (359,225.24) .. (359,228) .. controls (359,230.76) and (356.76,233) .. (354,233) .. controls (351.24,233) and (349,230.76) .. (349,228) -- cycle ;
%Flowchart: Connector [id:dp19277088648275997] 
\draw  [fill={rgb, 255:red, 0; green, 0; blue, 0 }  ,fill opacity=1 ] (138,183) .. controls (138,180.24) and (140.24,178) .. (143,178) .. controls (145.76,178) and (148,180.24) .. (148,183) .. controls (148,185.76) and (145.76,188) .. (143,188) .. controls (140.24,188) and (138,185.76) .. (138,183) -- cycle ;
%Straight Lines [id:da8305040114176477] 
\draw    (143,10) -- (143,302) ;
%Straight Lines [id:da9758237315618791] 
\draw    (188,10) -- (188,302) ;
%Straight Lines [id:da6123060551392492] 
\draw    (238,10) -- (238,302) ;
%Straight Lines [id:da18886186828355866] 
\draw    (290,10) -- (290,302) ;
%Flowchart: Connector [id:dp748863336567992] 
\draw  [fill={rgb, 255:red, 0; green, 0; blue, 0 }  ,fill opacity=1 ] (183,147) .. controls (183,144.24) and (185.24,142) .. (188,142) .. controls (190.76,142) and (193,144.24) .. (193,147) .. controls (193,149.76) and (190.76,152) .. (188,152) .. controls (185.24,152) and (183,149.76) .. (183,147) -- cycle ;
%Flowchart: Connector [id:dp39225373961156207] 
\draw  [fill={rgb, 255:red, 0; green, 0; blue, 0 }  ,fill opacity=1 ] (232,116) .. controls (232,113.24) and (234.24,111) .. (237,111) .. controls (239.76,111) and (242,113.24) .. (242,116) .. controls (242,118.76) and (239.76,121) .. (237,121) .. controls (234.24,121) and (232,118.76) .. (232,116) -- cycle ;
%Flowchart: Connector [id:dp6113417787023882] 
\draw  [fill={rgb, 255:red, 0; green, 0; blue, 0 }  ,fill opacity=1 ] (284,89) .. controls (284,86.24) and (286.24,84) .. (289,84) .. controls (291.76,84) and (294,86.24) .. (294,89) .. controls (294,91.76) and (291.76,94) .. (289,94) .. controls (286.24,94) and (284,91.76) .. (284,89) -- cycle ;
%Straight Lines [id:da08463905023390095] 
\draw    (87,229) -- (426,227) ;
%Straight Lines [id:da938235010903678] 
\draw    (88,184) -- (427,182) ;
%Straight Lines [id:da7302896665705999] 
\draw    (88,148) -- (427,146) ;
%Straight Lines [id:da307438384840403] 
\draw    (88,117) -- (427,115) ;

% Text Node
\draw (99,309) node [anchor=north west][inner sep=0.75pt]    {$1\ {\cdots \ j_{1} \cdots \ \ j_{2} \ \cdots \ j_{m}{}_{-1} \ \cdots\ j_{m} \ \cdots \ k+1\ \cdots \ \ \ n\ }$};
% Text Node
\draw (149,159) node [anchor=north west][inner sep=0.75pt]    {$\pi _{j_{1}}$};
% Text Node
\draw (359,201) node [anchor=north west][inner sep=0.75pt]    {$\pi _{k+1}$};
% Text Node
\draw (42,17) node [anchor=north west][inner sep=0.75pt]    {$ \begin{array}{l}
\ \ \ \ \ n\\
\ \ \\
\\
\\\\
\\
\\
\ \ \ \ \ \vdots \\
\ \\
\\
\\
\\
\\
\\
\ \ \ \ \ 1
\end{array}$};
% Text Node
\draw (192,123) node [anchor=north west][inner sep=0.75pt]    {$\pi _{j}{}_{_{2}}$};
% Text Node
\draw (243,92) node [anchor=north west][inner sep=0.75pt]    {$\pi _{j_{m-1}}$};
% Text Node
\draw (295,65) node [anchor=north west][inner sep=0.75pt]    {$\pi _{j_{m}}$};
% Text Node
\draw (360,159) node [anchor=north west][inner sep=0.75pt]   [align=left] {Region 2};
% Text Node
\draw (359,58) node [anchor=north west][inner sep=0.75pt]   [align=left] {Region 3};
% Text Node
\draw (189,256) node [anchor=north west][inner sep=0.75pt]   [align=left] {\begin{minipage}[lt]{43.55pt}\setlength\topsep{0pt}
\begin{center}
Region 1
\end{center}

\end{minipage}};

\end{tikzpicture}
         }
         \caption{Structure of case 4. A shaded region implies that that region can have no points $(i,\pi_i)$ other than the ones included.}
         \label{incOrder}
     \end{figure}
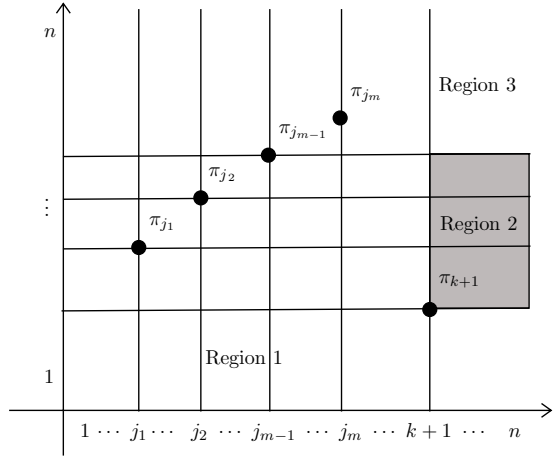
\end{enumerate}
    Collectively, we have now shown that in each case, car $\pi_{k+1}$ prefers independently when the cars $\pi_1,\pi_2,\pi_3, \ldots, \pi_k$ all prefer independently. This completes the inductive step and hence the proof.
\end{proof}

Proposition~\ref{inequality} and Proposition~\ref{equality} imply the following result.

\begin{theorem}\label{thm:equality}
Let $n\in\mathbb{N}$. Then 
 \[|\OMVPn^{-1}(\pi)|=\prod_{j=1}^{n}|\Omega_{j}(\pi)|\] if and only if $\pi$ avoids $(3,2,1)$ and $(3,4,1,2)$.
\end{theorem}
  
\begin{remark}\label{remark:AlexWoo}
The set of permutations that avoid the patterns $(3,2,1)$ and $(3,4,1,2)$ were studied by Tenner in the context of the Boolean algebra, the set of subsets of $[n]$ ordered by inclusion \cite[Theorem 4.3]{Tenner}. 
Moreover, the number of permutations  in $\Sym_n$
that avoid $(3,2,1)$ and $(3,4,1,2)$ 
permutations is
$F_{2n-1}$, where $F_k$
denotes the $k$th
Fibonacci number\footnote{OEIS \textcolor{blue}{\href{https://oeis.org/A000045}{A000045}}.}, for more details see \cite{Fan,West}. More recently, Lee, Masuda, and Park provide a summary of the relations of this result to algebraic geometry \cite[Theorem 1.1]{LMP}.  
\end{remark}

\subsection{Applications: $k$-cycles}
In this section we consider $k$-cycles with decreasing or increasing consecutive entries. 
We begin by recalling that a $k$-cycle is a permutation consisting of a single cycle of length~$k$. 
We first consider \emph{increasing k-cycles}.
\begin{definition}
Let $\pi \in \Sym_n$.
We say $\pi$ is an \emph{increasing $k$-cycle} if, in cycle notation\footnote{Note that we use ``$\langle$'' and ``$\rangle$'' to denote the cycle notation of a permutation since we use parenthesis to denote the one-line notation.}, it has the form $\langle a,a+1,\ldots, a+k-1\rangle$, for some $a\in[n]$.
\end{definition}

In one-line notation, these permutations can be described as having entries in increasing order and in which one entry $a\in[n]$ has been moved right by $k-1$ spots. 
Namely, in one-line notation an increasing $k$-cycle has the form:
\begin{align}
\pi=(1,2,\ldots,a-1,\hat{a},a+1,\ldots,a+k-1,a,a+k,\ldots,n),\label{eq:move a right}
\end{align}
where we shift the smaller entry $a$ by $k-1$ indices to the right and each of the entries $a+1,a+2,\ldots,a+k-1$ to the left by one index, while all other entries remain in place. 
Informally, we think of this as just ``moving a small number to the right.'' Note that in \eqref{eq:move a right} we write $\hat{a}$ to illustrate that we have removed this instance of $a$ from the permutation.

\begin{example}\label{ex:kcycles}
The permutation $\pi =\langle2,3,4\rangle=(1,3,4,2,5,6)$ is an increasing $3$-cycle with $a=2$, whereas $\tau =\langle1,2,3,4,5\rangle= (2,3,4,5,1)$ is an increasing $5$-cycle with $a=1$.
\end{example}

\begin{lemma}\label{lem:kcycles}
If $\pi=\langle a, a+1,\ldots, a+k-1\rangle=(1,2,\ldots,a-1,\hat{a},a+1,\ldots,a+k-1,a,a+k,\ldots,n)$ is an increasing $k$-cycle, then
\[|\OMVPn^{-1}(\pi)|=k.\]
\end{lemma}
\begin{proof}
We begin by noting that $\pi$ avoids $(3,2,1)$ and $(3,4,1,2)$.
Therefore, we can apply Theorem~\ref{thm:equality}, i.e., preference independence holds in this case. 
Note that for any car $c$ with $c \neq a$, car $c$ can only prefers one spot by Definition~\ref{prefSpots}, namely the spot in which it parks. 
Now, consider car $a$.
Since we move car $a$ to the right, all the cars left of car $a$ and right of car $a-1$ are greater than $a$. 
By Definition~\ref{prefSpots}, car $a$ can prefer any of the $k-1$ spots occupied by these cars. 
By Definition~\ref{prefSpots}, $a$ can also prefer the spot in which it parks.
Therefore, the total number of possible preferences for car $a$ is $k$. 
Since all other cars only prefer one spot, then, the product of all possible preferences is equal to~$k$.
\end{proof}

Applying Lemma~\ref{lem:kcycles} to the permutations of Example~\ref{ex:kcycles} yields:
\[|\OMVPn^{-1}((1,3,4,2,5,6))| = 3\qquad\mbox{and}\qquad|\OMVPn^{-1}((2,3,4,5,1))| = 5. \]

We now turn our attention to \emph{decreasing $k$-cycles}.
\begin{definition}
Let $\pi \in \Sym_n$. 
We say $\pi$ is a \emph{decreasing $k$-cycle} if, in cycle notation, it has the form $\langle b,b-1,\ldots, b-k+1\rangle$, for some $b\in[n]$. 
\end{definition}
In one-line notation, these permutations can be described as having entries in increasing order and in which one entry $b\in[n]$ has been moved left by $k-1$ spots. Namely, in one-line notation a decreasing $k$-cycle has the form:
\begin{align}\pi=(1,2,\ldots,b-k,b,b-k+1,\ldots,b-1,\hat{b},b+1,\ldots,n),\label{eq:move b left}\end{align}
where we shift the larger entry $b$ by $k-1$ indices to the left and each of the entries $b-k+1,b-k+2,\ldots,b-1$ to the right by one index, while all other entries remain in place. 
Informally, we think of this as just ``moving a larger number to the left.'' 
As before, in \eqref{eq:move b left}, we write $\hat{b}$ to illustrate that we have removed this instance of $b$ from the permutation.

\begin{example}\label{ex:inckcycles}
The permutation $\pi =\langle6,5,4\rangle= (1,2,3,6,4,5)$ is an decreasing $3$-cycle with $b = 6$, whereas the permutation $\tau = \langle5,4,3,2,1\rangle=(5,1,2,3,4)$ is an increasing $5$-cycle with  $b = 5$.
\end{example}

We now establish the following.
\begin{lemma}\label{lem:inckcycles}
Let $2\leq b\leq n$. If $\pi=\langle b,b-1,\ldots,b-k+1\rangle=(1,2,\ldots,b-k,b,b-k+1,\ldots,b-1,\hat{b},b+1,\ldots,n)$ 
is a decreasing $k$-cycle, then
\[|\OMVPn^{-1}(\pi)|=2^{k-1}.\]
\end{lemma}
\begin{proof}
We begin by noting that $\pi$ avoids $(3,2,1)$ and $(3,4,1,2)$.
Therefore, we can apply Theorem~\ref{thm:equality}, i.e., preference independence holds in this case. 
Note that all of the cars left of $b$ are smaller than $b$, and that they are arranged in increasing order up to $b$. 
Therefore, by Definition~\ref{prefSpots}, each these cars can only prefer one spot, namely the spot in which it parks.  
Similarly, by Definition~\ref{prefSpots}, car $b$ can only prefers one spot, namely the spot in which it parks. 
Lastly, cars $b+1,\ldots,n$ can only prefer one spot as well, by the same argument. 
Now consider all the cars right of $b$ and left of $b+1$, which are the cars $b-k+1,b-k+2,\ldots,b-1$. These cars are arranged in increasing order, but are each less than $b$. 
Therefore, by Definition~\ref{prefSpots}, each of these cars can only prefers two spots: the one they currently occupy and that of $b$. 
The total number of such cars is $k-1$ and since each car prefers two spots independently, the product of their possible preferences is given by $2^{k-1}$. 
Since all other cars prefer only one spot, the product of all possible preferences is equal to $2^{k-1}$.
\end{proof}

Applying Lemma~\ref{lem:inckcycles} to the permutations of Example~\ref{ex:inckcycles} yields:
\[|\OMVPn^{-1}((1,2,3,6,4,5))| = 2^{2-1}= 4\qquad \mbox{and}\qquad |\OMVPn^{-1}((5,1,2,3,4)) |= 2^{5-1}= 16.\]

\section{The outcome map and Motzkin paths}
\label{sec:special perms}
In this section, we give a bijection between Motzkin paths of length $n$ and the elements of $\OMVPn^{-1}(w_0)$, where $w_0\coloneqq(n,n-1,\ldots,2,1)$ is the longest word in $\Sym_n$. 

\begin{definition}
Let $n \in \mathbb{N}$.
A \emph{Motzkin path} of length $n$ is a lattice path consisting solely of horizontal steps $(1,0)$, upward diagonal steps $(1,1)$, and downward diagonal steps $(1,-1)$, which begins at $(0,0)$ and ends at $(n,0)$, and which does not fall below the $x$-axis. 
\end{definition}
Figure~\ref{fig:Motzkin path} illustrates the nine Motzkin paths of length $4$.

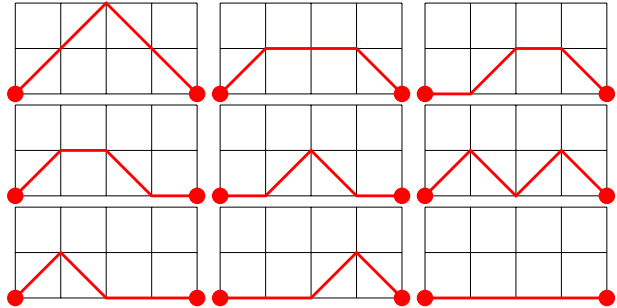
\begin{figure}[!ht]
    \centering
    \resizebox{.5\textwidth}{!}{
\begin{tikzpicture}[scale=.9, transform shape] \tikzstyle{every node} = [circle, fill=red]
\draw (0,0) grid (4,2);
 \node (a) at (0, 0) {};
 \node (b) at (4,0) {};
\draw[ultra thick, red] (0,0)--(2,2)--(4,0);
\end{tikzpicture}

\begin{tikzpicture}[scale=.9, transform shape] \tikzstyle{every node} = [circle, fill=red]
\draw (0,0) grid (4,2);
 \node (a) at (0, 0) {};
 \node (b) at (4,0) {};
\draw[ultra thick, red] (0,0)--(1,1)--(3,1)--(4,0);
\end{tikzpicture}
\begin{tikzpicture}[scale=.9, transform shape] \tikzstyle{every node} = [circle, fill=red]
\draw (0,0) grid (4,2);
 \node (a) at (0, 0) {};
 \node (b) at (4,0) {};
\draw[ultra thick, red] (0,0)--(1,0)--(2,1)--(3,1)--(4,0);
\end{tikzpicture}
}
\resizebox{.5\textwidth}{!}{
\begin{tikzpicture}[scale=.9, transform shape] \tikzstyle{every node} = [circle, fill=red]
\draw (0,0) grid (4,2);
 \node (a) at (0, 0) {};
 \node (b) at (4,0) {};
\draw[ultra thick, red] (0,0)--(1,1)--(2,1)--(3,0)--(4,0);
\end{tikzpicture}

\begin{tikzpicture}[scale=.9, transform shape] \tikzstyle{every node} = [circle, fill=red]
\draw (0,0) grid (4,2);
 \node (a) at (0, 0) {};
 \node (b) at (4,0) {};
 \draw[ultra thick, red] (0,0)--(1,0)--(2,1)--(3,0)--(4,0);
\end{tikzpicture}

\begin{tikzpicture}[scale=.9, transform shape] \tikzstyle{every node} = [circle, fill=red]
\draw (0,0) grid (4,2);
 \node (a) at (0, 0) {};
 \node (b) at (4,0) {};
 \draw[ultra thick, red] (0,0)--(1,1)--(2,0)--(3,1)--(4,0);
\end{tikzpicture}
}
\resizebox{.5\textwidth}{!}{
\begin{tikzpicture}[scale=.9, transform shape] \tikzstyle{every node} = [circle, fill=red]
\draw (0,0) grid (4,2);
 \node (a) at (0, 0) {};
 \node (b) at (4,0) {};
 \draw[ultra thick, red] (0,0)--(1,1)--(2,0)--(3,0)--(4,0);
\end{tikzpicture}

\begin{tikzpicture}[scale=.9, transform shape] \tikzstyle{every node} = [circle, fill=red]
\draw (0,0) grid (4,2);
 \node (a) at (0, 0) {};
 \node (b) at (4,0) {};
 \draw[ultra thick, red] (0,0)--(1,0)--(2,0)--(3,1)--(4,0);
\end{tikzpicture}
\begin{tikzpicture}[scale=.9, transform shape] \tikzstyle{every node} = [circle, fill=red]
\draw (0,0) grid (4,2);
 \node (a) at (0, 0) {};
 \node (b) at (4,0) {};
 \draw[ultra thick, red] (0,0)--(1,0)--(2,0)--(3,0)--(4,0);
\end{tikzpicture}
}
    \caption{Motzkin paths of length $4$.}
    \label{fig:Motzkin path}
\end{figure} 

We let $\mathcal{M}_n$ denote the set of Motzkin paths of length $n \in \mathbb{N}$.
The sequence $|\mathcal{M}_n|$ for $n \in \mathbb{N}$ is known as the \textit{Motzkin numbers}\footnote{OEIS \textcolor{blue}{\href{https://oeis.org/A001006}{A001006}}.}.
The sequence begins:
\[1, 2, 4, 9, 21, 51, 127, 323, 835, 2188, 5798, 15511, 41835,\ldots\]
and has closed formula 
\[|\mathcal{M}_n|=\sum_{k=0}^{\lfloor\frac n2\rfloor}\binom{n}{2k}C_k,\]
where $C_k=\frac{1}{k+1}\binom{2k}{k}$ denotes the $k$th Catalan number\footnote{OEIS 
\textcolor{blue}{\href{http://oeis.org/A000108}{A000108}}. For a comprehensive survey, see Stanley~\cite{stanley2015catalan}.
}.

For $n \in \mathbb{N}$, let $\mathcal{P}_n$ be the set of lattice paths of length $n$ beginning at $(0,0)$ and consisting of upward diagonal, downward diagonal, and horizontal steps.
We begin by constructing lattice paths from MVP parking functions as follows.
\begin{definition}
\label{lattice}
Let $n \in \mathbb{N}$ and $v = (v_1, v_2, \ldots, v_n) \in [n]^n$.
Let $\Phi: [n]^n \rightarrow \mathcal{P}_n$, where $\Phi(v)=\phi(1)\phi(2)\cdots\phi(n)$ starts from $(0,0)$ and is built iteratively as follows.
For each $j \in [n]$:
\begin{itemize}
    \item If $|\{i \in [n] : v_i = j \}| = 0$, then $\phi(j)=D$ denoting a downward diagonal step $(1,-1)$.
    \item If $|\{i \in [n] : v_i = j \}| = 1$,  then $\phi(j)=H$ denoting a horizontal step $(1,0)$.
    \item If $|\{i \in [n] : v_i = j \}| \geq 2$,  then $\phi(j)=U$ denoting an upward diagonal step $(1,1)$.
\end{itemize}
\end{definition}

\begin{example}
Figure~\ref{fig:path} illustrates the lattice path corresponding to $v = (2,2,1,3)$, which one constructs based on the following: 
\begin{itemize}
    \item If $j=1$ or $j=3$, then $|\{i\in[4]:v_i=1\}|=|\{i\in[4]:v_i=3\}|=1$, so $\phi(1)=\phi(3)=H$.
    \item If $j=2$, then $|\{i\in[4]:v_i=2\}|=2$, so $\phi(2)=U$.
    \item If $j=4$, then $|\{i\in[4]:v_i=4\}|=0$, so $\phi(4)=D$.
\end{itemize}
Hence $\Phi((2,2,1,3))=HUHD$. Note that this is the same lattice path as that corresponding to $v = (2,1,2,3)$.
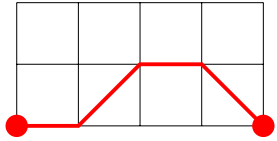
\begin{figure}[!ht]
    \centering
        \resizebox{1.5in}{!}{
\begin{tikzpicture}[scale=.9, transform shape] \tikzstyle{every node} = [circle, fill=red]
\draw (0,0) grid (4,2);
 \node (a) at (0, 0) {};
 \node (b) at (4,0) {};
\draw[ultra thick, red] (0,0)--(1,0)--(2,1)--(3,1)--(4,0);
\end{tikzpicture}
}
    \caption{Lattice path corresponding to $(2,2,1,3)$ and $(2,1,2,3)$.}
    \label{fig:path}
\end{figure} 
\end{example}

The main result of this section is as follows.    
\begin{theorem}\label{thm:motzkin}
If $n\geq 1$ and $w_0 = (n,n-1,\ldots,3,2,1) \in \Sym_n$ (i.e., $w_0$ is the longest word in $\Sym_n$), then, $|\OMVPn^{-1}(w_0)|= |\mathcal{M}_n|$.
\end{theorem}
In order to prove Theorem~\ref{thm:motzkin} 
we establish a bijection between the set of MVP parking functions $\alpha$ satisfying $\OMVPn(\alpha)=(n,n-1,\ldots, 3,2,1)$ and the set of lattice paths $\Phi(\alpha)$ arising from them, and then we show that the set $\Phi(\alpha)$ is precisely the set $\M_n$.
To begin we show that at most two cars can prefer the same spot if $\alpha$ satisfies $\OMVPn(\alpha)=(n,n-1,\ldots,3,2,1)$.

\begin{lemma}\label{twoOnes}
If $\alpha$ is an MVP parking function satisfying $\OMVPn(\alpha)=(n,n-1,\ldots, 3,2,1)$, then any spot in $[n]$ is preferred by at most two cars.
\end{lemma}

\begin{proof}
Assume for the sake of contradiction that, for some $i \in [n]$, there are more than 2 cars that prefer spot $i$ in $\alpha$. Since car $n-i+1$ must park at spot $i$, it must be the last car preferring this spot. Consider three cars $x$, $y$, $z$ that all prefer spot $i$, where $x < y < z$, and in which car $z = n-i+1$, and no car $y < y' < z$ prefers spot $i$. In this situation, car $x$ is bumped by car $y$ (or some other car) out of spot $i$, and then car $y$ is bumped out of spot $i$ by car $z$. Since any car $z'$ after $z$ must park at some spot $j < i$, $z'$ does not interfere with how cars $x, y, z$ park, so our MVP parking rule yields $\OMVPn(\alpha) = (\ldots, z, \ldots, x, \ldots, y, \ldots)$. 
We have arrived at a contradiction, since the three cars are supposed to park in the order $\ldots z \ldots y \ldots x \ldots$.
\end{proof}

For the rest of this section, we restrict the domain of $\Phi$ to the set  $\OMVPn^{-1}((n,n-1,\ldots,3,2,1))$.
\begin{lemma}\label{lem:down}
If $\alpha$ is an MVP parking function satisfying $\OMVPn(\alpha)=(n,n-1,\ldots, 3,2,1)$, then for any $i \in [n]$ the lattice path $\Phi(\alpha)$ satisfies that the first $i$ steps have at least as many upward diagonals as downward diagonals.
\end{lemma}
\begin{proof}
For all $j \in [n]$, parking spot $j$ corresponds to the $j$th lattice step in $\Phi(\alpha)$. By Lemma~\ref{twoOnes} and by definition of $\Phi(\alpha)$, 
each upward diagonal corresponds to a spot preferred by exactly two cars, each horizontal step corresponds to a spot preferred by exactly one car, and each downward diagonal corresponds to a spot preferred by no cars. For each $j \in [n]$, car $n-j+1$ must prefer a spot in $[j]$, since it must park at spot $j$.

Let $i \in [n]$. Among the first $i$ steps, label those that are downward diagonals as $i_1, i_2, \ldots, i_m \leq i$. For each $j \in [m]$, car $n-i_j+1$ must park at spot $i_j$, but prefers some spot $i'_j < i_j$ that is not a downward diagonal. Thus, we have $i$ cars preferring the remaining $i-m$ spots. By Lemma~\ref{twoOnes} and the pigeonhole principle, at least $m$ of these remaining $i-m$ spots must be preferred exactly twice. Thus, there are at least $m$ upward diagonals among the first $i$ steps.
\end{proof}

\begin{corollary}\label{start at (0,0)}
If $\alpha$ is an MVP parking function satisfying $\OMVPn(\alpha)=(n,n-1,\ldots, 3,2,1)$, then the lattice path $\Phi(\alpha)$ always begins at $(0,0)$ and the first step is either a horizontal step or an upward diagonal step.
\end{corollary}
\begin{proof}
We know from Definition~\ref{lattice} that we always start at $(0,0)$. Now, from Lemma~\ref{lem:down}, we know that we can never start with a downward diagonal. Therefore, our first step is always going to be a horizontal step, or an upward diagonal.
\end{proof}

\begin{lemma}\label{lem:up}
If $\alpha$ is an MVP parking function satisfying $\OMVPn(\alpha)=(n,n-1,\ldots, 3,2,1)$, then for any $i \in [n]$ the lattice path $\Phi(\alpha)$ satisfies that the last $i$ steps have at least as many downward diagonals as upward diagonals.
\end{lemma}
\begin{proof}
Assume for the sake of contradiction that there exists $i \in [n]$ such that the last $i$ steps of $\Phi(\alpha)$ have $m_1$ downward diagonals and $m_2$ upward diagonals where $m_1 < m_2$. Among these last $i$ steps, there are $m_3 = i-m_1-m_2$ horizontal ones. By Lemma~\ref{twoOnes}, each of the $m_2$ spots that are upward diagonals is preferred by exactly two cars. It follows that the number of cars preferring the last $i$ spots is $2m_2+m_3 = m_2+m_3+m_2 = i-m_1+m_2 > i$. Thus, at least one of these cars cannot park in these $i$ spots, contrary to $\alpha$ being an MVP parking function.
\end{proof}

We now establish that the lattice paths constructed via $\Phi(\alpha)$ when $\alpha$ satisfies $\OMVPn(\alpha)=(n,n-1,\ldots,3,2,1)$ are indeed the set of Motzkin paths. Since the paths $\Phi(\alpha)$, by definition, begin at $(0,0)$ and consist of upward diagonal, downward diagonal, and horizontal steps, it suffices to establish that these paths end at $(n,0)$ and never fall below the $x$-axis. We establish this result next.

\begin{proposition}\label{parking to motzkin}
If $\alpha$ satisfies $\OMVPn(\alpha)=(n,n-1,\ldots,3,2,1)$, then $\Phi(\alpha) \in \mathcal{M}_n$.
\end{proposition}
\begin{proof}
By Lemma~\ref{lem:down}, $\Phi(\alpha)$ never falls below the $x$-axis as $\Phi(\alpha)$ has at least as many upward diagonals as downward ones at each step $i\in[n]$. Taking $i=n$ yields that $\Phi(\alpha)$ has at least as many upward diagonals as downward ones. By Lemma~\ref{lem:up}, $\Phi(\alpha)$ has at least as many downward diagonals as upward ones among the last $i$ steps, for every $i\in[n]$. Taking $i=n$ yields that $\Phi(\alpha)$ has at least as many downward diagonals as upward ones. Thus, $\Phi(\alpha)$ has the same number of upward diagonals as downward ones, ensuring that $\Phi(\alpha)$ starts at $(0,0)$ and ends at $(n,0)$ and never falls below the $x$-axis. Thus $\Phi(\alpha)\in\M_n$ as desired.
\end{proof}

For $n \in \mathbb{N}$, we now give the inverse map $\Phi^{-1}: \mathcal{M}_n \rightarrow \OMVPn^{-1}((n,n-1,\ldots,1))$, from Motzkin paths of length $n$ to MVP parking functions with reverse order outcome $(n,n-1,\ldots,2,1)$.

Following the convention of parking functions, by $i \mapsto j$ we will mean that car $i$ prefers spot $j$, or we can view it as ``adding'' car $i$ to spot $j$. Let $\pi = (n,n-1,\ldots,2,1)$ be the desired outcome permutation. Given a Motzkin path $P$ of length $n$, label these $n$ steps 1, 2, $\ldots$, $n$ corresponding to the $n$ parking spots. From left to right, we can determine what cars prefer each spot (or what cars to ``add'' to each spot), depending on whether that spot is a horizontal line (exactly one car), upward diagonal (exactly two cars), or downward diagonal (exactly zero cars). 
A \emph{nondecreasing spot} is either a horizontal line or an upward diagonal.

We now give a procedure for bracketing (or pairing) the upward diagonals of $P$ with the downward diagonals of $P$ which we will we use shortly in our analysis. Treat an upward diagonal as a left parenthesis ``('' and a downward diagonal as a right parenthesis ``)''. By the end of this procedure, all the upward diagonals will be paired with all the downward diagonals. For each step $S$ of the path $P$, let $\phi_S$ denote the spot corresponding to $S$. We proceed as follows:
\begin{enumerate}
\item Find the rightmost upward diagonal $U$ and pair it with the nearest downward diagonal $D$ to its right. Output the pairs $(U,D)$ and $(\phi_U,n-\phi_D+1)$. Then remove this pair of diagonals from consideration.
\item Repeat this process until all diagonals have been paired.
\end{enumerate}
\begin{remark}
In what follows, the pair $(\phi_U,n-\phi_D+1)$ will tell us that car $n-\phi_D+1$ prefers spot $\phi_U$ in order to park at spot $\phi_D$.
\end{remark}

\begin{definition}\label{mvp construction}
Let $P$ be a Motzkin path with $n$ steps. The \emph{MVP parking function corresponding to $P$} is constructed as follows: First note that car $i$ must park in spot $n-i+1$, due to the reverse order of $w_0=(n,n-1,\ldots,2,1)$, and it must prefer a spot in $[n-i+1]$. We fill the spots $1, 2, \ldots, n$ with the cars $n, n-1, \ldots, 1$,  
as follows:
\begin{enumerate}
\item Identify the nondecreasing steps $u_1, u_2, \ldots, u_k$ in $P$.\label{step1}
\item Add car $n-u_i+1$ to spot $u_i$ for each $i \in [k]$.\label{step2}
\item Add the remaining cars to the spots that are upward diagonals via the bracketing procedure for diagonals: For each resulting pair $(U,D)$ of diagonals, add car $n-\phi_D+1$ to spot $\phi_U$.\label{step3}
\end{enumerate}
\end{definition}

To display all the diagonals together with their corresponding spots, we put them in a two-line matrix whose upper row consists of these diagonals and whose lower row consists of their spots. Those positions not appearing in the bottom row of the matrix correspond to horizontal steps.

\begin{example}\label{enc exm}
Consider the Motzkin path $P$ of length $9$ illustrated in Figure~\ref{fig:p1}. \\
\begin{figure}[h!]
\centering\begin{tikzpicture}[scale=.6, transform shape] \tikzstyle{every node} = [circle, fill=red]
\draw (0,0) grid (9,3);
 \node (a) at (0, 0) {};
 \node (b) at (9,0) {};
\draw[ultra thick, red] (0,0)--(1,1)--(2,2)--(3,2)--(4,1)--(5,0)--(6,1)--(7,2)--(8,1)--(9,0);
\end{tikzpicture}
\caption{}\label{fig:p1}
\end{figure}

\noindent By Step \eqref{step1} we identify the nondecreasing steps: $u_1=1$, $u_2=2$, $u_3=3$, $u_4=6$, $u_5=7$. 

\noindent By Step \eqref{step2} we add car  $n-u_i+1$ to spot $u_i$
 which yields:
 \[9-1+1 = 9 \mapsto 1,\quad
9 - 2 + 1 = 8 \mapsto 2,\quad
9 - 3 + 1 = 7 \mapsto 3,\quad
9 - 6 + 1 = 4 \mapsto 6,\quad
9 - 7 + 1 = 3 \mapsto 7.
\]

Now we note that the diagonals of $P$ are: \[\left[ \begin{array}{cccccccc}
            U & U & D & D & U & U & D & D\\
            1 & 2 & 4 & 5 & 6 & 7 & 8 & 9
\end{array} \right]\]

\noindent In Step \eqref{step3} we apply the bracketing process between the steps $U$ and $D$ which results in the output $(\phi_U,n-\phi_D+1)$ and where we let $\hat{U}$ and $\hat{D}$ mean that those steps have been deleted from the path:
\begin{itemize}
\item $UUDDU(UD)D$ outputs pair $(7,9-8+1)=(7,2)$,
\item $UUDD(U\hat{U}\hat{D}D)$ outputs pair $(6,9-9+1)=(6,1)$,
\item $U(UD)D\hat{U}\hat{U}\hat{D}\hat{D}$ outputs pair $(2,9-4+1)=(2,6)$, and
\item $(U\hat{U}\hat{D}D)\hat{U}\hat{U}\hat{D}\hat{D}$  outputs pair $(1,9-5+1)=(1,5)$.
\end{itemize}
Finally, we add car $n-\phi_D+1$ to spot $\phi_U$:
 \[2 \mapsto 7,\quad
 1\mapsto 6,\quad
 6 \mapsto 2,\quad
 5\mapsto 1.\]
Thus, the corresponding MVP parking function is $\Phi^{-1}(P) = (6,7,7,6,1,2,3,2,1)$, which one can readily verify satisfies $\OMVPn(6,7,7,6,1,2,3,2,1)=(9,8,7,6,5,4,3,2,1)$.
\end{example}

\begin{example}\label{enc exm 2}
Consider the Motzkin path $P$ of length $9$ illustrated in Figure~\ref{fig:p2}. \\
\begin{figure}[h!]
\begin{tikzpicture}[scale=.6, transform shape] \tikzstyle{every node} = [circle, fill=red]
\draw (0,0) grid (9,4);
 \node (a) at (0, 0) {};
 \node (b) at (9,0) {};
\draw[ultra thick, red] (0,0)--(1,1)--(2,2)--(3,2)--(4,3)--(5,2)--(6,1)--(7,2)--(8,1)--(9,0);
\end{tikzpicture}
\caption{}\label{fig:p2}
\end{figure}

\noindent By Step \eqref{step1} we identify the nondecreasing steps: $u_1=1$, $u_2=2$, $u_3=3$, $u_4=4$, $u_5=7$. 

\noindent By Step \eqref{step2} we add car  $n-u_i+1$ to spot $u_i$
 which yields:
 \[9-1+1 = 9 \mapsto 1,\quad
9 - 2 + 1 = 8 \mapsto 2,\quad
9 - 3 + 1 = 7 \mapsto 3,\quad
9 - 4 + 1 = 6 \mapsto 4,\quad
9 - 7 + 1 = 3 \mapsto 7.
\]

Now we note that the diagonals of $P$ are:
\[\left[ \begin{array}{cccccccc}
            U & U & U & D & D & U & D & D\\
            1 & 2 & 4 & 5 & 6 & 7 & 8 & 9
\end{array} \right]\]

\noindent In Step \eqref{step3} we apply the bracketing process between the steps $U$ and $D$ which results in the output $(\phi_U,n-\phi_D+1)$ and where we let $\hat{U}$ and $\hat{D}$ mean that those steps have been deleted from the path:
\begin{itemize}
\item $UUUDD(UD)D$  outputs pair $(9-3+1,2)=(7,2)$
\item $UU(UD)D\hat{U}\hat{D}D$  outputs pair $(9-6+1,5)=(4,5)$
\item $U(U\hat{U}\hat{D}D)\hat{U}\hat{D}D$  outputs pair $(9-8+1,4)=(2,4)$
\item $(U\hat{U}\hat{U}\hat{D}\hat{D}\hat{U}\hat{D}D)$  outputs pair $(9-9+1,1)=(1,1)$
\end{itemize}
Finally, we add car $n-\phi_D+1$ to spot $\phi_U$:
 \[2 \mapsto 7,\quad
 5\mapsto 4,\quad
 4 \mapsto 2,\quad
 1\mapsto 1.\]

Thus, the corresponding MVP parking function is $\Phi^{-1}(P) =(1,7,7,2,4,4,3,2,1) $, which one can readily verify satisfies $\OMVPn=(9,8,7,6,5,4,3,2,1)$.
\end{example}

\begin{lemma}\label{bracketing works}
The bracketing procedure on $P$ results in the pairing of the upward diagonals with the downward diagonals in a one-to-one correspondence.
\end{lemma}
\begin{proof}
Since $P$ is a Motzkin path, its last $i$ steps have at least as many downward diagonals as upward ones, for any $i \in [n]$. It follows that the bracketing procedure pairs every upward diagonal with exactly one downward diagonal in an injective manner. Since the upward diagonals are equinumerous with the downward ones, the two sets are paired in a one-to-one correspondence by the bracketing procedure.
\end{proof}

For an upward diagonal $U$ paired with downward one $D$ via bracketing, we call the path connecting $U$ and $D$ the \emph{enclosure} of $(U,D)$, denoted $\mathrm{enc}(U,D)$; $\mathrm{enc}(U,D)$ will also denote the path connecting $\phi_U$ and $\phi_D$ (i.e. all the spots from $\phi_U$ to $\phi_D$), when the context is clear.

No diagonal enclosed by a bracketed pair $(U,D)$ of diagonals is paired with a diagonal outside the enclosure.
\begin{lemma}\label{pairs contained}
Suppose the upward diagonal $U$ is paired with the downward diagonal $D$ by the bracketing procedure on $P$, with the associated output pair $(\phi_U,n-\phi_D+1)$. If $U'$ and $D'$ are also paired diagonals with $U' \neq U$ and $D' \neq D$, then one of the following is true:
\begin{enumerate}
\item $U'$ is right of $U$ and $D'$ is left of $D$.
\item $U'$ is left of $U$ and $D'$ is right of $D$.
\item Both $U', D'$ are left of $U$.
\item Both $U', D'$ are right of $D$.
\end{enumerate}
\end{lemma}
\begin{proof}
It suffices to show that the following cases cannot occur:
\begin{enumerate}
\item $U'$ is between $U$ and $D$ while $D'$ is right of $D$.
\item $U'$ is left of $U$ while $D'$ is between $U$ and $D$.
\end{enumerate}
Suppose Case 1 occurs. Since $U'$ is right of $U$, $U'$ gets paired first. Hence $D$ has not yet been paired at the time $U'$ is getting paired, but $D$ is closer to $U'$ than $D'$ is. Thus, $D'$ is not paired with $U'$, which is a contradiction.

Suppose Case 2 occurs. Since $U$ is right of $U'$, $U$ gets paired first. Hence  $D'$ has not yet been paired at the time $U$ is getting paired, but $D'$ is closer to $U$ than $D$ is. Thus, $D$ is not paired with $U$, which is a contradiction.
\end{proof}
\begin{remark}
Lemma~\ref{pairs contained} shows that each pair of enclosures are either disjoint or related by containment.
This structural property is often referred to as ``laminarity'' (see~\cite[Chapter~13.4]{schrijver2003combinatorial}).
\end{remark}

An enclosure $\mathrm{enc}(U,D)$ is \emph{maximal} if no other enclosure $\mathrm{enc}(U',D')$ contains it, in other words there exist no paired diagonals $(U',D')$ with $U'$ left of $U$ and $D'$ right of $D$. 
In Example~\ref{enc exm}, there are two maximal enclosures--one with output pair $(6,1)$ and one with output pair $(1,5)$.
In Example~\ref{enc exm 2}, there is one maximal enclosure with output pair $(1,1)$.

\begin{corollary}\label{max enclosure decomp}
The Motzkin path $P$ can be partitioned into a sequence of maximal enclosures linked by horizontal steps.
\end{corollary}

For any enclosure $\mathrm{enc}(U,D)$, we define its \emph{laminar level} to be the length of any maximal sequence of enclosures $\mathrm{enc}(U,D)$, $\mathrm{enc}(U_1,D_1)$, $\mathrm{enc}(U_2,D_2)$, $\ldots$, $\mathrm{enc}(U_k,D_k)$ such that
\begin{enumerate}
\item $\mathrm{enc}(U,D)$ contains $\mathrm{enc}(U_1,D_1)$
\item $\mathrm{enc}(U_i,D_i)$ contains $\mathrm{enc}(U_{i+1},D_{i+1})$ for all $i \in [k-1]$.
\end{enumerate}
In Example~\ref{enc exm 2}, the enclosure corresponding to the output pair $(1,1)$ has laminar level 3.

Next, we show that a car with a smaller index preferring an upward diagonal $U$ always parks at the downward diagonal $D$ paired with $U$.

\begin{lemma}\label{upward parks downward}
Suppose the upward diagonal $U$ is paired with the downward diagonal $D$ by the bracketing procedure on $P$, with the associated output pair $(\phi_U,n-\phi_D+1)$. 
Then the cars preferring $\mathrm{enc}(U,D)$ are $n-\phi_D+1, n-\phi_D+2, \ldots, n-\phi_U+1$. Furthermore, by the time car $n-\phi_U+2$ enters the parking lot, car $n-\phi_D+i$ will have parked at its final spot $\phi_D-i+1$, for all $i \in [\phi_D-\phi_U+1]$.
\end{lemma}
\begin{proof}
We proceed via strong induction on the laminar level of $\mathrm{enc}(U,D)$. For the base case of laminar level 1, $U$ and $D$ are linked by horizontal steps only. The spots $\phi_U, \phi_U+1, \ldots, \phi_D-1$ are nondecreasing, so by definition they are preferred by cars $n-\phi_U+1, n-\phi_U, \ldots, n-\phi_D+2$, respectively. Car $n-\phi_D+1$ also prefers spot $\phi_U$, but it is bumped by car $n-\phi_U+1$ after the other cars park at the spots they prefer, and hence car $n-\phi_D+1$ parks at spot $\phi_D$. These cars park without interference from the cars preferring spots left of $\phi_U$, as the cars from the latter set that do the bumping are greater than $n-\phi_U+1$.

Now suppose that the claim is true for all enclosures of laminar level at most $k$. We prove the claim for $\mathrm{enc}(U,D)$ of laminar level $k+1$. The steps between $U$ and $D$ can be partitioned into a sequence, in right-to-left order, of disjoint enclosures $E_1, E_2, \ldots, E_l$ linked by horizontal steps, where $E_i$ has laminar level at most $k$ for all $i \in [l]$. Again, cars $n-\phi_U+1$ and $n-\phi_D+1$ both prefer spot $\phi_U$, with the cars between $n-\phi_U+1$ and $n-\phi_D+1$ preferring the spots between $\phi_U$ and $\phi_D$. By the inductive hypothesis, $E_i$ is completely and correctly parked before the cars preferring $E_{i+1}$ start entering the parking lot, for each $i \in [l-1]$; recall that $E_{i+1}$ is left of $E_i$. 
Also, for any horizontal step $j$, car $n-j+1$ parks at its preferred spot $j$. 
It follows that all the steps between $U$ and $D$ are completely and correctly parked before car $n-\phi_U+1$ enters the parking lot. 
Finally, car $n-\phi_D+1$ is bumped by car $n-\phi_U+1$ to its final spot $\phi_D$. 
Again, the bumping cars preferring spots left of $\phi_U$ are greater than $n-\phi_U+1$, so they do not interfere.
\end{proof}

\begin{proposition}
\label{inv mot}
If $P \in \mathcal{M}_n$, then $\Phi^{-1}(P) \in \OMVPn^{-1}((n,n-1,\ldots,2,1))$. 
\end{proposition}
\begin{proof}
By Corollary~\ref{max enclosure decomp}, we can partition the steps of $P$ into a sequence of maximal enclosures $E_1, E_2, \ldots, E_k$ linked by horizontal steps. Any car that gets bumped before parking must prefer one of these enclosures. By Lemma~\ref{upward parks downward}, the cars preferring $E_i$ will park correctly within $E_i$, for all $i \in [k]$. Finally, the cars preferring the horizontal steps will park where they prefer, since no bumping occurs at any horizontal step.
\end{proof}

Now that we have established both $\Phi$ and $\Phi^{-1}$, we obtain the following.
\begin{theorem}\label{motzkin thm}
The map $\Phi: \OMVPn^{-1}((n,n-1,\ldots,2,1)) \rightarrow \mathcal{M}_n$ is a bijection.
\end{theorem}
Theorem~\ref{thm:motzkin} immediately follows as a corollary.

\subsection{Applications: Permutations with an increasing tail}
We now consider a special subset of permutations in $\Sym_n$ for which we can apply Theorem~\ref{thm:motzkin} to describe the size of the fibers of the outcome map.
To do this, we formally describe a permutation ``with an increasing tail.''

\begin{definition}\label{def:pi'}
    Let $m > n$. For $\pi \in \Sym_n$ define the permutation $\pi'\in\Sym_m$ as follows:
    \[\pi' = \pi (n+1)(n+2)\cdots m.
    \]
\end{definition}

For example,
if $\pi = 4231\in\Sym_4$, then $\pi'=42315678\in\Sym_8$, and
    if $\pi = 21543\in\Sym_5$, then $\pi'=2154367\in\Sym_7$.
We can now establish our first result.

\begin{lemma}\label{lem:same size}
For $m>n$, if $\pi\in\Sym_n$ and $\pi'=\pi (n+1)(n+2)\cdots m\in\Sym_m$, as in Definition~\ref{def:pi'}, then $|\OMVPm^{-1}(\pi')| = |\OMVPn^{-1}(\pi)|.$
\end{lemma}

\begin{proof}
For this proof, it suffices to provide a bijection between $\OMVPm^{-1}(\pi')$ and $\OMVPn^{-1}(\pi)$. 

Let $\pi' = \pi (n+1)(n+2)\cdots m$ where $\pi \in \Sym_n$ and $m \geq n$. Notice that by Definition~\ref{prefSpots}, each car $\pi'_{n+1}, \ldots, \pi'_m$ can only prefer one spot, that is, where it parks and none of them prefer any spots in $[n]$. Similarly, by Definition~\ref{prefSpots}, the cars in $\pi$ cannot prefer spots $n+1, n+2,\ldots, m$. This implies that cars $\pi'_{n+1}, \ldots ,\pi'_m$ park independently of the cars in $\pi$.
We now define the map $\psi: \OMVPn^{-1}(\pi) \rightarrow \OMVPm^{-1}(\pi')$ given by
\[\psi(\beta) = \beta (n+1)(n+2)\cdots m.\]
This means that each car $i$, with $n+1 \leq i \leq m$, prefers spot $i$. 
We now show that $\psi$ is a one-to-one map. Consider $\alpha,\beta \in \OMVPn^{-1}(\pi)$ where $\psi(\alpha) = \psi(\beta)$. Then note 
\begin{align*}
    \psi(\alpha) &=\alpha (n+1)\cdots m = 
    \beta (n+1)\cdots m =\psi(\beta)
\end{align*}
if and only if $\alpha=\beta$, as expected.

We now show that the map $\psi$ is onto, that is that given $\beta\in\OMVPm^{-1}(\pi')$, there exists $\alpha\in\OMVPn^{-1}(\pi)$ such that $\psi(\alpha)=\beta$. 
To begin let
$\beta=(b_1,b_2,\ldots,b_m)\in\OMVPm^{-1}(\pi') $, where $\pi'=\pi(n+1)(n+2)\cdots m$ with $\pi\in\Sym_n$ arbitrary. 
By definition $\OMVPm(\beta)=\pi(n+1)(n+2)\cdots m$.
Let $ \alpha=(b_1,b_2,\ldots, b_n)$. By definition of $\beta$, note that 
$\OMVPn(\alpha)=\pi$. Thus $\alpha\in\OMVPn^{-1}(\pi)$. Now observe that by definition
$\psi(\alpha)=(b_1,b_2,\ldots, b_n,n+1,n+2,\ldots,m)$. We now claim that $\psi(\alpha)=\beta$. It suffices to show that $b_j=j$ for all $n+1\leq j\leq m$, but this is precisely the fact we established earlier, as car $c=n+1,\ldots,m$ can only prefer the spot $c$.
\end{proof}

From Lemma~\ref{lem:same size} and Theorem~\ref{thm:motzkin} we immediately have the following. 

\begin{corollary}
If $\pi'=w_0(n+1)(n+1)\cdots(m-1)m\in\Sym_n$ with $w_0\in\Sym_n$, then 
$|\OMVPn^{-1}(\pi')|=|\M_n|$.
\end{corollary}

\section{Future work}\label{sec:futurework}
In our work, we give a complete characterization of when the number of MVP parking functions that park in a given order satisfies parking independence. Using that result we considered permutations that are increasing and decreasing $k$-cycles, and gave closed formulas for the size of each outcome map fiber. 
Then we considered permutations which did not satisfy the pattern avoidance requirement of Theorem~\ref{thm:equality}, including the longest word and permutations with an increasing tail, cases in which we showed that the number of MVP parking functions parking in that order is a Motzkin number.
One could consider other families of permutations in order to give new formulas for the fiber sizes.

Moreover, we note that Theorem~\ref{thm:size>=1} implies that the smallest fiber corresponds to the identity permutation. 
In Table \ref{tab:large}, we provide computational evidence that the longest word $w_0 = (n, n - 1, \ldots, 2, 1)$ achieves the maximum fiber size (given by Motzkin numbers) for $1\leq n\leq 6$. 
However, the permutation $\widetilde{w}_0 \coloneqq (n - 1, n, n - 2, n - 3, \ldots, 2, 1)$ has a larger fiber size as soon as $n=7$, and this is the largest fiber among all permutations also for $n=8$.
\begin{table}[h!]
    \centering
    \begin{tabular}{|l|c|c|c|c|c|c|c|c|}\hline
         $n$&  1&2&3&4&5&6&7&8\\\hline
         $\OMVPn^{-1}(w_0)$& 1&2&4&9&21&51&127&323\\[2pt]\hline
         $\OMVPn^{-1}(\widetilde{w}_0)$&--&--& 3& 8& 20& 51& 131& 341\\[2pt]\hline
    \end{tabular}
    \caption{Largest fibers for $n=1,2,\ldots,8$}
    \label{tab:large}
\end{table}

For $n\geq 9$, it remains an open problem to determine a characterization for a permutation $\pi\in\Sym_n$ satisfying
    \[|\OMVPn^{-1}(\pi)|\geq |\OMVPn^{-1}(\tau)|\]  for all $\tau\in\Sym_{n}$.

\section*{Acknowledgments}
The authors thank Casandra D.~Monroe
for math conversations during the progress of this project.
We also thank Alexander Woo for the connection to Lehmer codes, the references and background provided in Remark~\ref{remark:AlexWoo}, and for insightful comments regarding the contents of this manuscript. The authors also thank MSRI and the Williams College Science Center for travel funding in support of this research.

\bibliographystyle{plain}
\bibliography{Bibliography.bib}

\end{document}